\documentclass[11pt]{article}
\usepackage{color}
\usepackage{enumerate}
\usepackage{amsthm,amsmath,amssymb}
\usepackage{vmargin}
\usepackage{hyperref}
\usepackage{gastex}
\usepackage{multirow}

\usepackage[T1]{fontenc} 
\usepackage[utf8]{inputenc}

\usepackage{algorithm}
\usepackage{algorithmic}

\setmarginsrb{3cm}{2cm}{3cm}{2cm}{0cm}{2cm}{0cm}{1cm}

\theoremstyle{plain}
\newtheorem{theorem}{Theorem}
\newtheorem{lemma}[theorem]{Lemma}
\newtheorem{corollary}[theorem]{Corollary}
\newtheorem{proposition}[theorem]{Proposition}

\theoremstyle{definition}
\newtheorem{definition}[theorem]{Definition}
\newtheorem{example}[theorem]{Example}
\newtheorem{open}[theorem]{Open Problem}

\newtheorem{remark}[theorem]{Remark} 

\newcommand{\N}{\mathbb{N}}
\newcommand{\lex}{\rm lex}
\newcommand{\Nyl}{\mathcal{N}}
\newcommand{\Lyn}{\mathcal{L}}
\newcommand{\NylF}{{\rm NylF}}
\newcommand{\NylC}{{\rm NylC}}
\newcommand{\length}{{\rm length}}

\DeclareMathOperator{\andrm}{\ {\rm and}\ }

\author{
\'Emilie Charlier\\
Department of Mathematics\\
University of Li\`ege\\
All\'ee de la D\'ecouverte 12\\
4000 Li\`ege, Belgium\\
\url{echarlier@uliege.be}\\
\and
Manon Philibert\\
Laboratoire d'Informatique et Systèmes \\
Aix-Marseille University\\
52 Av. Escadrille Normandie Niemen\\
13397 Marseille, France\\
\url{manon.philibert@lis-lab.fr}
\and
Manon Stipulanti\footnote{Corresponding author.}\\
Department of Mathematics\\
University of Li\`ege\\
All\'ee de la D\'ecouverte 12\\
4000 Li\`ege, Belgium\\
\url{m.stipulanti@uliege.be}}

\title{Nyldon words}  
\date{\today} 

\begin{document}
\thispagestyle{empty}
\maketitle

\begin{abstract} 
The Chen-Fox-Lyndon theorem states that every finite word over a fixed alphabet can be uniquely factorized as a lexicographically nonincreasing sequence of Lyndon words. This theorem can be used to define the family of Lyndon words in a recursive way. If the lexicographic order is reversed in this definition, we obtain a new family of words, which are called the Nyldon words. In this paper, we show that every finite word can be uniquely factorized into a lexicographically nondecreasing sequence of Nyldon words. Otherwise stated, Nyldon words form a complete factorization of the free  monoid with respect to the decreasing lexicographic order. Then we investigate this new family of words. In particular, we show that Nyldon words form a right Lazard set.
\end{abstract}

\bigskip
\hrule
\bigskip



\bigskip
\hrule
\bigskip

\section{Introduction}

The Chen-Fox-Lyndon theorem states that every finite word $w$ can be uniquely factorized as $w=\ell_1 \ell_2\cdots \ell_k$ where $\ell_1,\ldots,\ell_k$ are Lyndon words such that $\ell_1\ge_{\lex}\cdots\ge_{\lex}\ell_k$. This theorem can be used to define the family of Lyndon words over some totally ordered alphabet in a recursive way: the letters are Lyndon; a finite word of length greater than one is Lyndon if and only if it cannot be factorized into a nonincreasing sequence of shorter Lyndon words. In a Mathoverflow post dating from November 2014, Grinberg defines a variant of Lyndon words, which he calls Nyldon words, by reversing the lexicographic order in the previous recursive definition: the letters are Nyldon; a finite word of length greater than one is Nyldon if and only if it cannot be factorized into a {\em nondecreasing} sequence of shorter Nyldon words~\cite{Grinberg2014}. The class of words so obtained is not, as one might first think, the class of maximal words in their conjugacy classes. Grinberg asks three questions:
\begin{enumerate}
	\item How many Nyldon words of each length  are there? 
	\item Is there an equivalent to the Chen-Fox-Lyndon theorem for Nyldon words? More precisely, is it true 				that for every finite word $w$, there exists a unique sequence $(n_1,\ldots,n_k)$ of Nyldon words 					satisfying $w=n_1\cdots n_k$ and $n_1 \le_{\lex} \cdots \le_{\lex} n_k$?
	\item Is it true that every primitive word admits exactly one Nyldon word in its conjugacy class, whereas any 				non primitive word has no such conjugate?
\end{enumerate}
In this paper, we discuss the properties of this new class of words in the more general context of the complete factorizations of the free monoid as introduced by Schützenberger in \cite{Schutzenberger1965} and later on extensively studied in \cite{Berstel-Perrin-Reutenauer2010,Lothaire1997,Melancon1992,Reutenauer1993,Viennot1978}. In particular, we show that each of Grinberg's questions has a positive answer. 

Another variant of Lyndon words has been recently studied in \cite{Felice2018}: the {\em inverse Lyndon words}. Similarly to the Nyldon words, which are studied in the present paper, the inverse Lyndon words are built from the decreasing lexicographic order: they are the words that are greater than any of their proper suffixes. Although it is true that Nyldon words are greater than all their Nyldon proper suffixes (this is Theorem~\ref{thm:suffix}), the families of Nyldon words and inverse Lyndon words do not coincide. For example, the family of inverse Lyndon words is prefix-closed whereas the family of Nyldon words is not. Another major difference between those two families of words is that the inverse Lyndon factorization of a word is not unique in general, while the Nyldon factorization is always unique (this is Theorem~\ref{thm:unicity}).

This paper has the following organization. Section~\ref{sec:def} contains the necessary background and definitions. In Section~\ref{sec:prefixes}, we briefly discuss the prefixes of Nyldon words. Then, in Section~\ref{sec:counting}, we show that the unicity of the Nyldon factorization implies that there are equally many Nyldon and Lyndon words of each length. In Section~\ref{sec:suffixes}, we focus on suffixes of Nyldon words. In particular, we show that Nyldon words are greater than all their Nyldon proper suffixes. This allows us to prove in Section~\ref{sec:unicity} that the Nyldon factorization of a word is unique. Then, in Section~\ref{sec:computing-Nyldon-fac}, we provide algorithms for computing the Nyldon factorization of a word and for generating the Nyldon words up to any given length. Next, in Section~\ref{sec:standard-fac}, we show how to obtain a standard factorization of the Nyldon words, similarly to the standard factorization of the Lyndon words. In Section~\ref{sec:complete-fac}, we recall Schützenberger's theorem on factorizations of the free monoid and we show that this result combined with the unicity of the Nyldon factorization implies the primitivity of Nyldon words, as well as the fact that each primitive conjugacy class contains exactly one Nyldon word. Then, in Section~\ref{sec:comparison}, we realize an in-depth comparison between the families of Nyldon and Lyndon words. In particular, we emphasize how the exclusive knowledge of the recursive definition of the Lyndon words allows us to recover some known properties of Lyndon words, in a similar way to what we  do for Nyldon words. However, we note that the proofs in the Lyndon case are not simple translations of the proofs from the Nyldon case. We also show that the standard factorization of Nyldon words introduced in Section~\ref{sec:standard-fac} can be understood as the analogue of the right standard factorization of the Lyndon words. In Section~\ref{sec:Lazard}, we show that the Nyldon words form a right Lazard factorization, but not a left Lazard factorization. Finally, in Section~\ref{sec:Melancon}, we apply Melançon's algorithm in order to compute the Nyldon conjugate of any primitive word. Along the way, we mention seven open problems.

\section{Preliminaries}
\label{sec:def}

Throughout the text, $A$ designates an arbitrary finite alphabet of cardinality at least $2$ and endowed with a total order $<$.  Furthermore, when we need to specify the letters in $A$,  we use the notation $A=\{{\tt 0,1,\ldots,m}\}$, with ${\tt 0<1<\cdots<m}$. We use the usual definitions (prefix, suffix, factor, etc.)\ and notation of combinatorics on words; for example, see \cite{Lothaire1997}. In particular, $A^*$ is the set of all finite words over the alphabet $A$ and $A^+$ is the set of all nonempty finite words over the alphabet $A$. The empty word is denoted by $\varepsilon$ and the length of a finite word $w$ is denoted by $|w|$. A prefix (resp., suffix, factor) $u$ of a word $w$ is {\em proper} if $u\ne w$. 
Moreover, we say that a sequence $(w_1,\ldots,w_k)$ of nonempty words over $A$ is a {\em factorization} of a word $w$ if $w=w_1\cdots w_k$. We also say that $k$ is the {\em length} of this factorization. Note that, with the previous notation, we have $w=\varepsilon$ if and only if $k=0$. In particular, if $w$ is a nonempty word then $k\ge1$. Recall that the \emph{lexicographic order on $A^*$}, which we denote by $<_{\lex}$, is a total order on $A^*$ induced by the order $<$ on the letters and defined as follows: $u <_{\lex } v $ either if $u$ is a proper prefix of $v$ or if there exist ${\tt i,j} \in A$ and $p\in A^*$ such that ${\tt i<j}$, $p{\tt i}$ is a prefix of $u$ and $p{\tt j}$ is a prefix of $v$. As usual, we write $u \le_{\lex} v$ if $u<_{\lex} v$ or $u=v$. Also recall that two words $x$ and $y$ are {\em conjugate} if they are circular permutations of each other: $x=uv$ and $y=vu$ for some words $u,v$. A word $x$ is a {\em power} if $x=u^k$ for some word $u$ and some integer $k\ge2$. In particular, the empty word $\varepsilon$ is considered to be a power. Finally, a word which is not a power is said to be \emph{primitive}.

\begin{definition}
\label{def:Lyndon}
A finite word $w$ over $A$ is \emph{Lyndon} if it is primitive and lexicographically minimal among its conjugates. We let $\Lyn$ denote the set of all Lyndon words over $A$.
\end{definition}

It is easily verified that a word $w$ is Lyndon if and only if $w\neq \varepsilon$ and, for all $u,v\in A^+$ such that $w=uv$, we have $w<_{\lex} vu$.

The following theorem is usually referred to as the Chen-Fox-Lyndon factorization theorem~\cite{Chen-Fox-Lyndon1958}, although these authors never formulated the result in this way. It was simultaneously obtained by {\v S}ir{\v s}ov \cite{Sirsov1958}. Also see \cite{Bokut-Chen-Li2013,Berstel-Perrin2007} for a discussion on the origins of this theorem.

\begin{theorem}[\cite{Chen-Fox-Lyndon1958,Sirsov1958}] 
\label{the:Lyndon-fac}
For every finite word $w$ over $A$, there exists a unique factorization $(\ell_1,\ldots,\ell_k)$ of $w$ into Lyndon words over $A$ such that $\ell_1\ge_{\lex}\cdots\ge_{\lex}\ell_k$.
\end{theorem}

This result allows us to (re)define the set of Lyndon words recursively.

\begin{corollary}
\label{cor:LisL}
Let $L$ be the set of words over $A$ recursively defined as follows: letters are in $L$; a word of length at least two belongs to $L$ if and only if it cannot be factorized into a lexicographically nonincreasing sequence of shorter words of $L$. Then $L=\Lyn$.
\end{corollary}

\begin{proof}
An easy induction on $n$ shows that $L\cap A^n=\Lyn\cap A^n$ for all $n\ge 1$.
\end{proof}

Now we consider the following class of words, which were baptized the {\em Nyldon words} by Grinberg~\cite{Grinberg2014}.

\begin{definition}
\label{def:Nyldon-recursive}
Let $\Nyl$ be the set of words over $A$ recursively defined as follows: letters are in $\Nyl$; a word of length at least two belongs to $\Nyl$ if and only if it cannot be factorized into a (lexicographically) nondecreasing sequence of shorter words of $\Nyl$. Otherwise stated, a word $w$ is in $\Nyl$ either if it is a letter, or if there does not exist any factorization $(n_1,\ldots,n_k)$ of $w$ into words in $\Nyl$ such that $\max\{|n_1|,\ldots,|n_k|\}<|w|$ and $n_1\le_{\lex}\cdots\le_{\lex}n_k$. The words of $\Nyl$ are called the  {\em Nyldon words}.  Moreover, any factorization $(n_1,\ldots,n_k)$ of $w$ into Nyldon words such that $n_1\le_{\lex}\cdots\le_{\lex}n_k$ is called a {\em Nyldon factorization} of $w$. 
\end{definition}

With this new vocabulary, we can say that a word $w$ is Nyldon if and only if its only Nyldon factorization is $(w)$. Otherwise stated, $k=1$ and $n_1=w$. On the opposite, a nonempty word is not Nyldon if and only if it admits a Nyldon factorization of length at least $2$. The fact that every finite word admits at least one Nyldon factorization is immediate from the definition of Nyldon words, but still worth being emphasized.

\begin{proposition}
\label{prop:existence}
For every finite word $w$ over $A$, there exists a factorization $(n_1,\ldots,n_k)$ of $w$ into Nyldon words over $A$ such that $n_1\le_{\lex}\cdots\le_{\lex}n_k$.
\end{proposition}

We illustrate the definition of Nyldon words in the case where $A=\{{\tt 0},{\tt 1}\}$. In what follows, we will refer to the Nyldon words over this alphabet as the {\em binary Nyldon words}.

\begin{example}
By Definition~\ref{def:Nyldon-recursive}, the letters ${\tt 0}$ and ${\tt 1}$ are Nyldon. The word ${\tt 00}$ (resp., ${\tt 01}$, ${\tt 11}$) of length $2$ admits the Nyldon factorization $({\tt 0,0})$ (resp., $({\tt 0 ,1})$, $({\tt 1, 1})$), and hence is not a Nyldon word. Consequently, the only binary Nyldon word of length $2$ is ${\tt 10}$. The binary Lyndon and Nyldon words up to length $7$ are stored in Table~\ref{table:first-Lyndon-Nyldon}. 
\begin{table}[h]
\[
	\begin{tabular}{l | l||l|l||l|l}
	\text{Lyndon} 	& \text{Nyldon} 	& \text{Lyndon} 	& \text{Nyldon} 	& \text{Lyndon} 	& \text{Nyldon}	\\
	\hline
	{\tt 0} 			& {\tt 0} 		& {\tt 000001}	& {\tt 100000}	& {\tt 0001011}		& {\tt 1000110} 	\\
	{\tt 1} 			& {\tt 1} 		& {\tt 000011}	& {\tt 100001}	& {\tt 0001101}	& {\tt 1000111}	\\
	{\tt 01} 			& {\tt 10} 		& {\tt 000101}	& {\tt 100010}	& {\tt 0001111}	& {\tt 1001010} 	\\
	{\tt 001}		& {\tt 100} 		& {\tt 000111}	& {\tt 100011}	& {\tt 0010011}	& {\tt 1001100}	\\
	{\tt 011}		& {\tt 101} 		& {\tt 001011}	& {\tt 100110}	& {\tt 0010101}	& {\tt 1001110} 	\\
	{\tt 0001}		& {\tt 1000} 		& {\tt 001101}	& {\tt 100111}	& {\tt 0010111}	& {\tt 1001111} 	\\
	{\tt 0011}		& {\tt 1001} 		& {\tt 001111}	& {\tt 101100}	& {\tt 0011011}	& {\tt 1011000} 	\\
	{\tt 0111}		& {\tt 1011} 		& {\tt 010111}	& {\tt 101110}	& {\tt 0011101}	& {\tt 1011001} 	\\
	{\tt 00001}		& {\tt 10000} 	& {\tt 011111}	& {\tt 101111}	& {\tt 0011111}	& {\tt 1011010} 	\\
	{\tt 00011}		& {\tt 10001} 	& {\tt 0000001}	& {\tt 1000000}	& {\tt 0101011}	& {\tt 1011100} 	\\
	{\tt 00101}		& {\tt 10010} 	& {\tt 0000011}	& {\tt 1000001}	& {\tt 0101111}	& {\tt 1011101}	\\ 
	{\tt 00111}		& {\tt 10011} 	& {\tt 0000101}	& {\tt 1000010}	& {\tt 0110111}	& {\tt 1011110} 	\\
	{\tt 01011}		& {\tt 10110} 	& {\tt 0000111}	& {\tt 1000011}	& {\tt 0111111}	& {\tt 1011111} 	\\
	{\tt 01111}		& {\tt 10111} 	& {\tt 0001001}	& {\tt 1000100}	& 				& 				 
	\end{tabular}
\]
\caption{The binary Lyndon and Nyldon words up to length $7$.}
\label{table:first-Lyndon-Nyldon}
\end{table}
\end{example}

Note that Nyldon words are not lexicographically extremal among their conjugates since, for instance, ${\tt 101}$ is a binary Nyldon.

\section{Prefixes of Nyldon words}
\label{sec:prefixes}

Apart from the letters, the binary Nyldon words of Table~\ref{table:first-Lyndon-Nyldon} all start with the prefix ${\tt 10}$. This fact is true for longer lengths as well as all alphabet sizes, as shown by the following proposition. However, note that both ${\tt 100}$ and ${\tt 101}$ are prefixes of Nyldon words (see Table~\ref{table:first-Lyndon-Nyldon}).

\begin{proposition}
\label{prop:prefixes}
Each Nyldon word over $\{{\tt 0},{\tt 1},\ldots,{\tt m}\}$ of length at least $2$ starts with ${\tt ij}$ with ${\tt 0\le j< i \le m}$. 
\end{proposition}

\begin{proof}
Let $w={\tt ij}u$, with $u\in\{{\tt 0},{\tt 1},\ldots,{\tt m}\}^*$ and ${\tt 0\le i\le j \le m}$. Our aim is to show that $w$ cannot be Nyldon. Let $(n_1,\ldots, n_k)$ be a Nyldon factorization of ${\tt j}u$. Then $n_1$ begins with ${\tt j}$. Since ${\tt i}$ is Nyldon, ${\tt i}\le_{\lex}{\tt j}\le_{\lex} n_1$ and $k\ge1$, we obtain that $({\tt i},n_1,\ldots,n_k)$ is a Nyldon factorization of $w$ of length at least $2$, whence $w$ is not Nyldon.
\end{proof}

For instance, all binary Nyldon of length at least $2$ start with the prefix ${\tt 10}$, and all ternary Nyldon words of length at least $2$ start with ${\tt 10}$, ${\tt 20}$ or ${\tt 21}$. Other prefixes than those of Propostion~\ref{prop:prefixes} are forbidden in the family of Nyldon words. We introduce the following definition.

\begin{definition}
We say that a finite word $p$ over $A$ is a {\em forbidden prefix} if no Nyldon word over $A$ starts with $p$.
\end{definition}

In the following proposition, we exhibit a family of forbidden prefixes by generalizing the argument of Propostion~\ref{prop:prefixes}.

\begin{proposition}
\label{prop:forbidden-pref}
All elements of the set
\begin{multline*}
	F= \{p_1p_2p_3\in A^*\colon 
							p_1\in\Nyl,\ 
							p_1\le_{\lex} p_2\ 
							\text{ and, for all } u\in A^*  \text{ and for all Nyldon}\\
							\text{factorizations } 
							(n_1,\ldots, n_k) \text{ of }  p_2p_3u, \text{ one has } |n_1|\ge |p_2|
							\}
\end{multline*}
are forbidden prefixes.
\end{proposition}

\begin{proof}
Let $p\in F$. Then it can be decomposed as $p=p_1p_2p_3$ where $p_1,p_2,p_3$ satisfy the properties described in the statement (see Figure~\ref{fig:forbidden-pref}). 
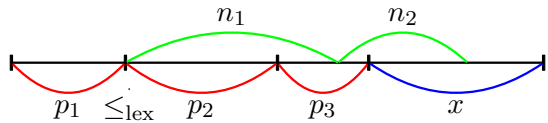
\begin{figure}[h]
\begin{picture}(55, 15)(-40,-10)
 	\gasset{Nw=.3,Nh=2,Nmr=0,fillgray=0} 
  	\gasset{ExtNL=y,NLdist=1,NLangle=-90} 
  		\node(A)(0,  0){}
  		\node(B)(15,0){}
		\node(C)(35,0){}
   		\node(D)(47,0){}
   		\node(F)(70,0){}
   
  	\gasset{Nw=0,Nh=0,Nmr=0,fillgray=0} 
  	\gasset{ExtNL=y,NLdist=1,NLangle=0} 
    		\node(E1)(43,0){}
	  	\node(E2)(60,0){}
		
  	\gasset{ExtNL=y,NLdist=1,NLangle=-90} 
    		\node(BC)(15.5,-3.5){$\le_{\lex}$}
    
    	\gasset{ExtNL=y,NLdist=1,NLangle=180} 
		\drawedge[linewidth=0.3, AHnb=0, ELside=l](A,F){}
               
	\gasset{curvedepth=-4}
		\drawedge[linecolor=red, linewidth=0.3, AHnb=0,ELside=r](A,B){$p_1$}
		\drawedge[linecolor=red, linewidth=0.3, AHnb=0,ELside=r](B,C){$p_2$}
		\drawedge[linecolor=red, linewidth=0.3, AHnb=0,ELside=r](C,D){$p_3$}
		\drawedge[linecolor=blue, linewidth=0.3, AHnb=0,ELside=r](D,F){$x$}

	\gasset{curvedepth=4}
		\drawedge[linecolor=green, linewidth=0.3, AHnb=0,ELside=l](B,E1){$n_1$}
		\drawedge[linecolor=green, linewidth=0.3, AHnb=0,ELside=l](E1,E2){$n_2$}
\end{picture}
\caption{A family of forbidden prefixes.}
\label{fig:forbidden-pref}
\end{figure}
We have to show that all words over $A$ starting with $p$ are not Nyldon. Let $w=pu$, with $u\in A^*$. Let $(n_1,\ldots ,n_k)$ be a Nyldon factorization of $p_2p_3u$. By definition of $F$,  $p_1\le_{\lex}p_2$ and $p_2$ must be a prefix of $n_1$, hence $p_1\le_{\lex} n_1$. Since we also have $p_1\in\Nyl$, we obtain that $(p_1,n_1,\ldots,n_k)$ is a Nyldon factorization of $w$ of length at least $2$, whence $w$ is not Nyldon.
\end{proof}

\begin{example}
We obtain from Proposition~\ref{prop:forbidden-pref} that, for all $k\in\N$, the words ${\tt 10}^k{\tt 10}^k$,  ${\tt 10}^k{\tt 1011}$, ${\tt 101}^{k+1}{\tt 01}^{k+1}$, ${\tt 10}^{k+2}{\tt 110}^{k+1}{\tt 11}$  are forbidden prefixes. The corresponding factors $p_1,p_2,p_3$ as in the definition of the set $F$ are given in Table~\ref{table:ex-F}.

\begin{table}[htbp]
\centering
\begin{tabular}{c|c|c|c}
$w$									& $p_1$					& $p_2$ 				& $p_3$ 		\\		
\hline
${\tt 10}^k{\tt 10}^k$					& ${\tt 10}^k$ 			& ${\tt 10}^k$ 			& $\varepsilon$	\\
${\tt 10}^k{\tt 1011}$					& ${\tt 10}^k$ 			& ${\tt 101}$ 			& ${\tt1}$		\\
${\tt 101}^{k+1}{\tt 01}^{k+1}$ 			& ${\tt 101}^k$			& ${\tt 101}^k$			& ${\tt 1}$ 		\\
${\tt 10}^{k+2}{\tt 110}^{k+1}{\tt 11}$ 	& ${\tt 10}^{k+2}{\tt 1}$	& ${\tt 10}^{k+1}{\tt 1}$	& ${\tt 1}$
\end{tabular}
\caption{Some forbidden prefixes.}
\label{table:ex-F}
\end{table}	

\end{example}

All the examples of forbidden prefixes that we have come from the family $F$ of Proposition~\ref{prop:forbidden-pref}. In general, it seems hard to understand the prefixes of Nyldon words, while we are able to understand  their suffixes quite well (see Section~\ref{sec:suffixes}). We leave the following unresolved question as an open problem.

\begin{open}
\label{op:sesqui}
Characterize the language of forbidden prefixes. In particular, prove or disprove that there are other forbidden prefixes than those given by the family $F$ of Proposition~\ref{prop:forbidden-pref}. 
\end{open}

As is well known, the nonempty prefixes of Lyndon words are exactly the sesquipowers of Lyndon words distinct of the maximal letter~\cite{Duval1983,Knuth2011}. A {\em sesquipower} of a word $x$ is a word $x^kp$ where $k$ is a positive integer and $p$ is a proper prefix of $x$. 
It is not true that all sesquipowers of Nyldon words distinct of the letters can be prefixes of Nyldon words since for example ${\tt 1010}$ is a forbidden prefix. Moreover, we know that each prefix of a Lyndon word is a sesquipower of exactly one Lyndon word \cite{Duval1983,Knuth2011}. This property does not hold for Nyldon words either since for example, the word ${\tt 101101}$ is the prefix of the Nyldon word ${\tt 1011011}$ and is a sesquipower of both the Nyldon words ${\tt 101}$ and ${\tt 10110}$. However, we do not know whether all prefixes of Nyldon words are sesquipowers of Nyldon words. 

\begin{open}
Prove or disprove that prefixes of Nyldon words must be sesquipowers of Nyldon words.
\end{open}

\section{Counting Nyldon words of length $n$}
\label{sec:counting}

As already noticed by Grinberg, if we can prove the unicity of the Nyldon factorization (which is actually proved in Section~\ref{sec:unicity}), then we know that there are equally many Lyndon and Nyldon words of each length. Indeed, suppose that every finite word admits a unique Nyldon factorization. We proceed by induction on the length $n$ of the words to show that $\#(\Nyl\cap A^n)=\#(\Lyn\cap A^n)$ for all $n\ge1$. Since all the letters are both Lyndon and Nyldon, the base case is verified. Now suppose that $n\ge2$ and that, for all $m<n$, there is the same number of Lyndon and Nyldon words of length $m$. Since we have assumed that the Nyldon factorization of each word is unique, the number of words of length $n$ that are not Nyldon is equal to the number of Nyldon factorizations $(n_1,\ldots,n_k)$ such that $|n_1|,\ldots,|n_k|<n$. Similarly, we know by Theorem~\ref{the:Lyndon-fac} that the Lyndon factorization of each word is unique, hence the number of words of length $n$ that are not Lyndon is equal to the number of Lyndon factorizations $(\ell_1,\ldots,\ell_k)$ such that $|\ell_1|,\ldots,|\ell_k|<n$. Then, using the induction hypothesis, we obtain that $\#(\Nyl^c\cap A^n)=\#(\Lyn^c\cap A^n)$, where $X^c$ denotes the complement of the set $X$. This, of course, implies that $\#(\Nyl\cap A^n)=\#(\Lyn\cap A^n)$ as well.

We illustrate this construction for the binary words of length $4$. Table~\ref{table:counting-Nyldon} is built as follows. The left column contains the non-Lyndon words of length $4$, ordered in decreasing lexicographic order. If the non-Lyndon word of a certain row admits the Lyndon factorization $(\ell_1,\ldots,\ell_k)$, then the corresponding non-Nyldon word in the same row has the Nyldon factorization $(n_1,\ldots,n_k)$, where $\{|\ell_1|,\ldots,|\ell_k|\}=\{|n_1|,\ldots,|n_k|\}$ and for each occurrence of a factor $\ell$ in the Lyndon factorization $(\ell_1,\ldots,\ell_k)$, if $\ell$ the $j$th word of length $|\ell|$ in the list of Lyndon words of length up to $3$ in decreasing lexicographic order, then,  in the corresponding Nyldon factorization $(n_1,\ldots,n_k)$, there is an occurrence of the $j$th word of the same length $|\ell|$  in the list of Nyldon words of length up to $3$ in increasing lexicographic order.
\begin{table}[htbp]
\centering
\begin{tabular}{c|c|c||c|c|c}
    \multicolumn{3}{c||}{Lyndon} & \multicolumn{3}{c}{Nyldon}\\
    \hline
  \multicolumn{3}{c||}{${\tt 1>_{\lex}011>_{\lex}01>_{\lex}001>_{\lex}0}$} & 
\multicolumn{3}{c}{${\tt 0<_{\lex}1<_{\lex}10<_{\lex}100<_{\lex}101}$}  \\
\hline
${\tt 1111}$	& ${\tt (1,1,1,1)}$ 	& $1+1+1+1$ 	& ${\tt 0000}$	& ${\tt (0,0,0,0)}$ 	& $1+1+1+1$ 	\\
${\tt 1110}$	& ${\tt (1,1,1,0)}$ 	& $1+1+1+1$ 	& ${\tt 0001}$	& ${\tt (0,0,0,1)}$ 	& $1+1+1+1$ 	\\
${\tt 1101}$	& ${\tt (1,1,01)}$ 	& $1+1+2$ 		& ${\tt 0010}$	& ${\tt (0,0,10)}$ 	& $1+1+2$ 		\\
${\tt 1100}$	& ${\tt (1,1,0,0)}$ 	& $1+1+1+1$	& ${\tt 0011}$	& ${\tt (0,0,1,1)}$ 	& $1+1+1+1$ 	\\
${\tt 1011}$	& ${\tt (1,011)}$ 		& $1+3$ 		& ${\tt 0100}$	& ${\tt (0,100)}$ 		& $1+3$ 		\\
${\tt 1010}$	& ${\tt (1,01,0)}$ 	& $1+2+1$ 		& ${\tt 0110}$	& ${\tt (0,1,10)}$ 	& $1+1+2$ 		\\
${\tt 1001}$	& ${\tt (1,001)}$ 		& $1+3$ 		& ${\tt 0101}$	& ${\tt (1,101)}$ 		& $1+3$ 		\\
${\tt 1000}$	& ${\tt (1,0,0,0)}$ 	& $1+1+1+1$ 	& ${\tt 0111}$	& ${\tt (0,1,1,1)}$	& $1+1+1+1$ 	\\
${\tt 0110}$	& ${\tt (011,0)}$ 		& $3+1$ 		& ${\tt 1100}$	& ${\tt (1,100)}$ 		& $1+3$ 		\\
${\tt 0101}$	& ${\tt (01,01)}$ 		& $2+2$ 		& ${\tt 1010}$	& ${\tt (10,10)}$ 		& $2+2$ 		\\
${\tt 0100}$	& ${\tt (01,0,0)}$ 	& $2+1+1$ 		& ${\tt 1110}$	& ${\tt (1,1,10)}$ 	& $1+1+2$ 		\\
${\tt 0010}$	& ${\tt (001,0)}$ 		& $3+1$ 		& ${\tt 1101}$	& ${\tt (1,101)}$ 		& $1+3$ 		\\
${\tt 0000}$	& ${\tt (0,0,0,0)}$ 	& $1+1+1+1$ 	& ${\tt 1111}$	& ${\tt (1,1,1,1)}$ 	& $1+1+1+1$ 
\end{tabular}
\caption{There are equally many non-Nyldon words and non-Lyndon words of length $4$.}
\label{table:counting-Nyldon}
\end{table}	

The sequence 
\[
	2, 1, 2, 3, 6, 9, 18, 30, 56, 99, 186,335, 630, 1161,\ldots
\]
that counts the number of binary Lyndon words of length $n\ge 1$ is referred to as A001037 in  Sloane's On-Line Encyclopedia of Integer Sequences.

\section{Suffixes of Nyldon words}
\label{sec:suffixes}

Unlike their prefixes, the suffixes of Nyldon words satisfy nice properties. In this section, we prove two results that will allow us to obtain the unicity of the Nyldon factorization.

\begin{theorem}\label{thm:suffix}
Let $w \in A^*$ be a Nyldon word. For each Nyldon proper suffix $s$ of $w$, we have $s <_{\lex} w$.
\end{theorem}

\begin{proof}
We proceed by induction on $|w|$. The case $|w|=1$ is obvious. If $|w|=2$, write $w={\tt ij}$ with ${\tt i,j}\in A$. The only Nyldon proper suffix of $w$ is its last letter ${\tt j}$ and ${\tt j <_{\lex}  i} <_{\lex} w$ by Proposition~\ref{prop:prefixes}.

Now, we suppose that $|w|\ge 3$ and that the result is true for all Nyldon words shorter than $w$. Proceed by contradiction and assume that there exists a Nyldon proper suffix $s$ of $w$ such that 
\begin{equation}
\label{eq: s and w}
	s \ge_{\lex} w.
\end{equation}
Among all such suffixes of $w$, we choose $s$ to be the longest and we write $w=ps$ with $p\in A^*$. Our goal is to reach the contradiction that this well-chosen suffix $s$ is not Nyldon by exhibiting a Nyldon factorization of $s$ of length at least $2$.

First, we show that $p\notin \Nyl$. Indeed, if $p \in \Nyl$ and $p \le_{\lex} s$, then $(p,s)$ is a Nyldon factorization of $w$ of length $2$, contradicting that $w$ is Nyldon. Thus, if $p$ were Nyldon, then $s<_{\lex} p<_{\lex} w$, which would contradict the assumption \eqref{eq: s and w} on $s$. Hence $p\notin\Nyl$.

Let $(p_1, \ldots, p_k)$ be a Nyldon factorization of $p$ of length $k\ge 2$. Since $w,s\in \Nyl$, we must have 
\begin{equation}
\label{eq:p_k and s}
	p_k >_{\lex} s,
\end{equation}
for otherwise $(p_1, \ldots, p_{k-1},p_k,s)$ would be a Nyldon factorization of $w$ of length $k+1\ge 2$. Moreover, $p_ks \notin \Nyl$ for otherwise $(p_1, \ldots, p_{k-1},p_ks)$ would be a Nyldon factorization of $w$ of length $k\ge 2$.

Now let $(n_1, \ldots, n_\ell)$ be a Nyldon factorization of $p_ks$ of length $\ell\ge 2$. We claim that there exist $i\in\{1,\ldots, \ell\}$ and $x,y\in A^+$ such that 
\[	
	n_i = xy,\quad p_k=n_1 \cdots n_{i-1} x \quad\text{ and }\quad s= y n_{i+1} \cdots n_\ell.
\]
Indeed, suppose instead that $p_k=n_1 \cdots n_j$ and $s=n_{j+1} \cdots n_\ell$ with $1\le j \le \ell-1$ (recall that $p_k$ and $s$ are not empty). But since both $p_k$ and $s$ are Nyldon, this implies $p_k=n_1$ and $s=n_2$ (and hence, $\ell=2$). But $n_1 \le_{\lex}  n_2$, which contradicts~\eqref{eq:p_k and s}. 

Next, we must have $i\le \ell-1$. Indeed, suppose that $i=\ell$. Then $y=s$ and $n_\ell=xs$. By induction hypothesis, we would get that $s <_{\lex} n_\ell$. Then, by using~\eqref{eq: s and w}, we obtain that $n_\ell$ is a Nyldon proper suffix of $w$ longer than $s$ and such that  $n_\ell>_{\lex}s\ge_{\lex}w$. This is impossible by maximality of the length of $s$.

Now let $(y_1, \ldots, y_t)$ be a Nyldon factorization of $y$ (with $t\ge 1$). Then $y_t$ is a Nyldon proper suffix of $n_i$, which is itself Nyldon. Since $|n_i| < |w|$, the induction hypothesis yields $y_t <_{\lex} n_i$. But then, since  $n_i \le_{\lex} n_{i+1}$, we obtain that $(y_1, \ldots, y_t,n_{i+1}, \ldots, n_\ell)$ is a Nyldon factorization of $s$ of length at least $2$, contradicting that $s$ is Nyldon as announced.
\end{proof}

\begin{theorem}
\label{thm:longest-Nyldon-suffix}
Let $w\in A^+$ and let $(n_1,\ldots,n_k)$ be a Nyldon factorization of $w$. Then $n_k$ is the longest Nyldon suffix of $w$. 
\end{theorem}

\begin{proof}
Let $s$ denote the longest Nyldon suffix of $w$. If $w$ is Nyldon, then $k=1$ and $s=w=n_k$. Suppose now that $w$ is not Nyldon. Then $k\ge 2$ and $s$ is a proper suffix of $w$. Write $w=ps$. By choice of $s$ and since $n_k$ is Nyldon, we have $|n_k| \le |s|$. Let us show that $|n_k|=|s|$. Proceed by contradiction and suppose that $|n_k| < |s|$. Since $s$ is Nyldon, we cannot have $s=n_i\cdots n_k$ with $i<k$. Therefore, there must exist $i \in \{1,\ldots,k-1\}$ and $x,y\in A^+$ such that $n_i=xy$, $p=n_1 \cdots n_{i-1} x$ and $s=y n_{i+1} \cdots n_k$.  Let $(y_1,\ldots,y_t)$ be a Nyldon factorization of $y$ (with $t\ge 1$). From Theorem~\ref{thm:suffix}, we deduce that $y_t <_{\lex} n_i$ since $y_t$ is a proper suffix of $n_i$. But then, since $n_i\le_{\lex} n_{i+1}$, we obtain that $(y_1,\ldots,y_t,n_{i+1},\ldots,n_k)$ is a Nyldon factorization of $s$ of length at least $2$, contradicting that $s$ is Nyldon. Consequently, $|n_k|=|s|$, which in turn implies $n_k=s$. 
\end{proof} 

However, it is not true that the longest Nyldon prefix of $w$ is the first factor of any Nyldon factorization, as illustrated below.

\begin{example}
Let $w={\tt 10100}$. Then $({\tt 10,100})$ is a Nyldon factorization of $w$ although its longest Nyldon prefix is ${\tt 101}$. 
\end{example}

\section{Unicity of the Nyldon factorization}
\label{sec:unicity}

Using Theorem~\ref{thm:longest-Nyldon-suffix}, we obtain that there can be only one Nyldon factorization of each word, which positively answers  the second question of Grinberg mentioned in the introduction (and hence the first as well in view of Section~\ref{sec:counting}). In particular, the Nyldon words form a complete factorization of the free monoid $A^*$, which we make explicit below.

\begin{theorem}
\label{thm:unicity}
For every finite word $w$ over $A$, there exists a unique sequence $(n_1,\ldots,n_k)$ of Nyldon words such that $w=n_1\cdots n_k$ and $n_1\le_{\lex}\cdots\le_{\lex}n_k$.
\end{theorem}

\begin{proof}
The existence of the Nyldon factorization is known from Proposition~\ref{prop:existence}. Let us show that the unicity follows from Theorem~\ref{thm:longest-Nyldon-suffix}. We proceed by induction on the length of the words. If $|w|\le1$, then the result is clear. Now suppose that $|w|\ge 2$ and that every finite word shorter than $w$ admits a unique Nyldon factorization. If $w$ is Nyldon, its only possible factorization is $(w)$. Suppose now that $w$ is not Nyldon and let $(n_1,\ldots,n_k)$ be a Nyldon factorization of $w$. Then $k\ge 2$ and $n_k$ is  the longest Nyldon suffix of $w$ by Theorem~\ref{thm:longest-Nyldon-suffix}. Since $(n_1,\ldots,n_{k-1})$ is a Nyldon factorization of a word shorter than $w$, the factors $n_1,\ldots,n_{k-1}$ are also completely determined by $w$ by induction hypothesis. Therefore, there cannot be another Nyldon factorization of $w$ than $(n_1,\ldots,n_k)$.
\end{proof}

From now on, thanks to Theorem~\ref{thm:unicity}, we will talk about {\em the} Nyldon factorization of a finite word $w$ instead of {\em a} Nyldon factorization (see Definition~\ref{def:Nyldon-recursive}) to refer to the unique sequence $(n_1,\ldots,n_k)$ of Nyldon words such that $w=n_1\cdots n_k$ and $n_1\le_{\lex}\cdots \le_{\lex}n_k$.

As mentioned in Section~\ref{sec:counting}, Theorem~\ref{thm:unicity} implies that there are equally many Lyndon and Nyldon words of each length. This result, however, does not provide us with a natural bijection between Lyndon and Nyldon words of the same length. Such a bijection will be given by Theorem~\ref{thm:conjugacy} below in Section~\ref{sec:complete-fac}. Although Table~\ref{table:counting-Nyldon} describes a length-preserving bijection between non-Lyndon and non-Nyldon words, it is not fully satisfying, in the sense that we need to precompute all Lyndon and Nyldon words up to length $n-1$ in order to compute the image of a non-Lyndon word of length $n$ under this bijection.

\section{A faster algorithm for computing the Nyldon factorization}
\label{sec:computing-Nyldon-fac}
The recursive definition of the Nyldon words comes with a natural algorithm for generating the Nyldon words up to any given length. Unfortunately, the complexity of this first algorithm is clearly prohibitive. In this section, we show how to rapidly compute the Nyldon factorization of a word, and in turn, rapidly generate the Nyldon words up to any given length.

\begin{proposition}
\label{prop:longest-Nyldon-suffix}
Let $w\in A^+$, $w=ps$ where $s$ is the longest Nyldon proper suffix of $w$ and let $(p_1,\ldots,p_k)$ be the Nyldon factorization of $p$. Then w is Nyldon if and only if $p_k>_{\lex}s$.
\end{proposition}

\begin{proof}
Let us show, equivalently, that $w$ is not Nyldon if and only if $p_k\le_{\lex}s$. If $p_k\le_{\lex}s$, then  $w$ is not Nyldon since its Nyldon factorization is $(p_1,\ldots,p_k,s)$, which is of length at least $2$. In order to prove the other direction, we now suppose that $w$ is not Nyldon. Let $(w_1,\ldots,w_\ell)$ be the Nyldon factorization of $w$. Then $\ell\ge2$. By Theorem~\ref{thm:longest-Nyldon-suffix}, we obtain that $w_\ell=s$. But then $(p_1,\ldots,p_k)$ and $(w_1,\ldots,w_{\ell-1})$ are both Nyldon factorizations of $p$. By unicity of the Nyldon factorization, we must have $p_k=w_{\ell-1}$.  Since $w_{\ell-1}\le_{\lex}w_{\ell}$, we obtain $p_k=w_{\ell-1}\le_{\lex}w_{\ell}=s$, as desired.
\end{proof}

The previous result combined with Theorem~\ref{thm:longest-Nyldon-suffix} provides us with an algorithm that rapidly computes the Nyldon factorization of a nonempty finite word. In Algorithm~\ref{algo:fac}, $w[i]$ denotes the $i$th letter of $w$ and $w[i,j]$ denotes the factor $w[i]\cdots w[j]$ of $w$. Further, $\NylF(i)$ designates the $i$th element of the list $\NylF$ while $\NylF(-1)$ denotes its last element.
\begin{algorithm}
\caption{Compute the Nyldon factorization.}
\begin{algorithmic}
\REQUIRE $w\in A^+$
\ENSURE NylF is the Nyldon factorization of $w$
\STATE  $n\leftarrow |w|$, $\NylF \leftarrow (w[n])$
\FOR{$i=1$ to $n-1$}
\STATE $\NylF \leftarrow (w[n-i],\NylF)$
\WHILE{$\length( \NylF) \ge 2  \andrm \NylF(1)>_{\lex}\NylF(2)$}
\STATE $\NylF\leftarrow (\NylF(1)\cdot \NylF(2),\,\NylF(3),\, \ldots,\,\NylF(-1))$
\ENDWHILE
\ENDFOR
\RETURN NylF
\end{algorithmic}
\label{algo:fac}
\end{algorithm}

\begin{proposition}
Algorithm~\ref{algo:fac} halts for every input $w\in A^+$ and outputs the Nyldon factorization of $w$. Moreover, the worst case complexity of Algorithm~\ref{algo:fac} is $O\big(\frac{n(n-1)}{2}\big)$ where $n$ is the length of the input $w$. 
\end{proposition}

\begin{proof}
Clearly, Algorithm~\ref{algo:fac} halts for every input $w\in A^+$ since it goes exactly $n-1$ times into the for loop and for each  value of $i$ in the $n-1$ steps of the for loop, it goes at most $i$ times into the while loop. Then, since each lexicographic comparison between two words and basic manipulations of lists can be realized in constant time, the worst case complexity of Algorithm~\ref{algo:fac} is $O(1+2+\cdots +n-1)=O\big(\frac{n(n-1)}{2}\big)$. 

Let us now prove that Algorithm~\ref{algo:fac} is correct. We claim that, for each step $i$ of the for loop, if the variable NylF contains the Nyldon factorization of the suffix $w[n-i+1,n]$ of length $i$ of $w$ when it enters the for loop, then it exits with the Nyldon factorization of the suffix $w[n-i,n]$ of length $i+1$ of $w$.
Since NylF is initialized with $(w[n])$, which is the Nyldon factorization of the suffix of length $1$ of $w$, the claim implies that, after the last iteration of the for loop, that is for $i=n-1$, NylF contains the Nyldon factorization of  $w$.

In order to prove the claim, we suppose that $1\le i<n$ and that, before the $i$th step of the for loop, NylF contains the Nyldon factorization of the suffix $w[n-i+1,n]$ of length $i$ of $w$, which we denote by $(u_1,\ldots,u_k)$. First, NylF is updated to $(w[n-i],u_1,\ldots,u_k)$, before entering the while loop. 

By Theorem~\ref{thm:longest-Nyldon-suffix}, we know that, for each $j\in\{1,\ldots,k\}$, $u_j$ is the longest Nyldon suffix of $u_1\cdots u_j$. Then, since $w[n-i]$ is just a letter, $u_j$ is also the  longest Nyldon proper suffix of $w[n-i]u_1\cdots u_j$. From Proposition~\ref{prop:longest-Nyldon-suffix}, we successively obtain  that, for every $j\in\{1,\ldots,k\}$, the word $w[n-i]u_1\cdots u_{j}$ is Nyldon if and only if $w[n-i]u_1\cdots u_{j-1}>_{\lex}u_{j}$. Two cases are now possible. 

Either there is some $j\in\{1,\ldots,k-1\}$ such that 
\begin{align*}
	w[n-i] 						& >_{\lex}u_1	\\  
	w[n-i]u_1					& >_{\lex}u_2	\\
								& \vdots 		\\ 
	w[n-i]u_1u_2\cdots u_{j-1} 	& >_{\lex}u_j 	\\
	w[n-i]u_1\cdots u_j 			& \le_{\lex}u_{j+1},
\end{align*} 
in which case the words $w[n-i]u_1\cdots u_{j'}$ are Nyldon for all $j'\in\{1,\ldots,j\}$ by Proposition~\ref{prop:longest-Nyldon-suffix}. In this case, Algorithm~\ref{algo:fac} tells us  to successively update NylF from $(w[n-i],u_1,\ldots,u_k)$ to 
\[
	(w[n-i]u_1,u_2,\ldots,u_k),\ (w[n-i]u_1u_2,u_3,\ldots,u_k),\, \ldots,  \ (w[n-i]u_1\cdots u_j,u_{j+1},\ldots,u_k).
\] 
The last value of NylF, which is $(w[n-i]u_1\cdots u_j,u_{j+1},\ldots,u_k)$, is thus the Nyldon factorization of the suffix $w[n-i,n]$ of length $i+1$ of $w$. 

The alternative is that we have 
\begin{align*}
	w[n-i]						& >_{\lex}u_1 	\\  
	w[n-i]u_1					& >_{\lex}u_2 	\\
								& \vdots 		\\ 
	w[n-i]u_1u_2\cdots u_{k-1}	& >_{\lex}u_k.
\end{align*} 
In this case, Algorithm~\ref{algo:fac} tells us  to successively update NylF from $(w[n-i],u_1,\ldots,u_k)$ to 
\[
	(w[n-i]u_1,u_2,\ldots,u_k),\ (w[n-i]u_1u_2,u_3,\ldots,u_k),\ldots,  \ (w[n-i]u_1\cdots u_k).
\] 
As before, by Proposition~\ref{prop:longest-Nyldon-suffix}, we obtain that $w[n-i,n]=w[n-i]u_1\cdots u_k$ is Nyldon, hence the last value of NylF corresponds to the Nyldon factorization of the suffix of length $i+1$ of $w$.
\end{proof}

We then obtain an algorithm generating all Nyldon words over any alphabet $A$. Here we suppose that we have access to an  already implemented function TestNyldon which returns True if the word is Nyldon and False otherwise. Such a function is easily implemented by combining Algorithm~\ref{algo:fac} and the fact that a nonempty finite word $w$ is not Nyldon if and only if its Nyldon factorization is of length at least two.

\begin{algorithm}
\caption{Generate the Nyldon words.}
\begin{algorithmic}
\REQUIRE $\ell\ge1$ and an alphabet $A$
\ENSURE Nyl is the list of all Nyldon words over $A$ up to length $\ell$
\STATE  Nyl $\leftarrow A$
\FOR{$i=2$ to $\ell$}
\STATE Testlist $\leftarrow$ list of all words of length $i$ over $A$
\STATE Select the elements in Testlist that pass TestNyldon and join the selected elements to Nyl
\ENDFOR
\RETURN Nyl
\end{algorithmic}
\end{algorithm}

Using the fact that the prefixes of Lyndon words are precisely sesquipowers of Lyndon words, Duval obtained a linear algorithm computing the Lyndon factorization \cite{Duval1983}. Also see Open Problem~\ref{op:sesqui}.

\begin{open}
Obtain a linear or pseudo-linear algorithm computing the Nyldon factorization. 
\end{open}

\section{Standard factorization}
\label{sec:standard-fac}

Another important consequence of Theorem~\ref{thm:unicity} is that it allows us to define a standard factorization of Nyldon words, similarly to the case of Lyndon words. Also see Section~\ref{sec:comparison} where we  more specifically discuss the standard factorization of the Nyldon words in relationship with the standard factorization of Lyndon words.

\begin{theorem}
\label{thm:standard-fac}
Let $w\in A^+$ and let $w=ps$ where $s$ is the longest Nyldon proper suffix of $w$. Then w is Nyldon if and only if $p$ is Nyldon and $p>_{\lex}s$.
\end{theorem}

\begin{proof}
The sufficient condition  directly follows from Proposition~\ref{prop:longest-Nyldon-suffix}. In order to prove the necessary condition, we suppose that $w$ is Nyldon. If $p$ is also Nyldon, then clearly $p>_{\lex}s$ for otherwise the Nyldon factorization of $w$ would be $(p,s)$. So we only have to prove that $p$ is Nyldon. Suppose to the contrary that $p$ is not Nyldon. Then its Nyldon factorization $(p_1,\ldots,p_k)$ is of length $k\ge 2$. Since $w$ is Nyldon, we must have $p_k>_{\lex}s$. Observe that $s$ is also the longest Nyldon proper suffix of $p_k s$. Therefore, we obtain from Proposition~\ref{prop:longest-Nyldon-suffix} that $p_ks$ is Nyldon. But then the Nyldon factorization of $w$ would be $(p_1,\ldots,p_ks)$, which is of length $k\ge2$. This contradicts that $w$ is Nyldon. 
\end{proof}

\begin{definition}
Let $w$ be a Nyldon word over $A$ and let $s$ be its longest Nyldon proper suffix. Then the {\em standard factorization} of $w$ is defined to be the pair $(p,s)$, where $p$ is such that $w=ps$. 
\end{definition}

By Theorem~\ref{thm:standard-fac}, if $(p,s)$ is the standard factorization of a Nyldon word $w$, then $p>_{\lex}s$. Observe that  a Nyldon word $w$ may have several factorizations of the form $(u,v)$ such that 
\begin{equation}
\label{eq:almost-standard}
u,v\in\Nyl \ \text{ and }\ u>_{\lex}v.
\end{equation} 
The standard factorization is thus a distinguished such factorization. 

\begin{example}
The word $w={\tt 1011101}$ is a binary Nyldon word. The two pairs $({\tt 1011},{\tt 101})$ and $({\tt 1011101},{\tt 1})$ provide factorizations of $w$ satisfying~\eqref{eq:almost-standard}. The first one is the standard factorization.
\end{example}

\section{Complete factorizations of the free monoid and primitivity}
\label{sec:complete-fac}

Theorem~\ref{thm:unicity} implies that the class of Nyldon words is a particular case of the complete factorizations of the free monoid $A^*$, as introduced by Schützenberger in \cite{Schutzenberger1965}.

\begin{definition}
\label{def:complete-fac}
Let $F$ be a subset of $A^+$ endowed with some total order $\prec$. 
An {\em $F$-factorization} of a word $w$ over $A$ is a factorization $(f_1, \ldots ,f_k)$ of $w$ into words in $F$ such that $f_1 \succeq \cdots \succeq f_k$. 
Further, such a set $F$ is said to be a \emph{complete factorization  of $A^*$} if each $w\in A^*$ admits a unique $F$-factorization.
\end{definition}

\begin{example}
\label{ex:complete-Lyndon}
Thanks to Theorem~\ref{the:Lyndon-fac}, the Lyndon words over $A$ ordered by $<_{\lex}$ form a complete factorization of $A^*$. Similarly, thanks to Theorem~\ref{thm:unicity}, the Nyldon words over $A$ ordered by $>_{\lex}$ also form a complete factorization of $A^*$.
\end{example}

In \cite{Schutzenberger1965},  Schützenberger proved a beautiful and deep theorem about factorizations of the free monoid. For our needs, we only state here a particular case of Schützenberger’s result.

\begin{theorem}[Schützenberger]
\label{thm:Schutzenberger65}
Let $F$ be a subset of $A^+$ endowed with some total order $\prec$. Then any two of the following three conditions imply the third.
\begin{itemize}
	\item[(i)] Each word over $A$ admits at least one $F$-factorization.
	\item[(ii)] Each word over $A$ admits at most one $F$-factorization.
	\item[(iii)] All elements of $F$ are primitive and each primitive conjugacy class of $A^+$ contains exactly one element of $F$.
\end{itemize}
\end{theorem}

The following result is a consequence of both Theorems~\ref{thm:unicity} and~\ref{thm:Schutzenberger65}. It positively answers the third question raised by Grinberg and mentioned in the introduction.

\begin{theorem}
\label{thm:conjugacy}
All Nyldon words are primitive and every primitive word admits exactly one Nyldon word in its conjugacy class.
\end{theorem}

\begin{proof}
The first two conditions of Theorem~\ref{thm:Schutzenberger65} are satisfied by Proposition~\ref{prop:existence} and Theorem~\ref{thm:unicity} if we choose $F$ to be the set $\Nyl$ of Nyldon words and the order $\prec$ to be $>_{\lex}$. Consequently, all Nyldon words are primitive and every primitive word admits exactly one Nyldon word in its conjugacy class by Theorem~\ref{thm:Schutzenberger65}.
\end{proof}

A natural length-preserving bijection from Nyldon words to Lyndon words is now easy to describe: with each Nyldon word, we associate the unique Lyndon word in its conjugacy class. The Lyndon conjugate of a primitive word is easy to compute since it is lexicographically minimal  among all conjugates. However, the reciprocal map which associates with a Lyndon word  its unique Nyldon conjugate is much more difficult to understand. An effective construction of the Nyldon conjugate of any primitive word will be given by Algorithm~\ref{algo:conjugate} in Section~\ref{sec:Melancon}.

\section{Nyldon words versus Lyndon words}
\label{sec:comparison}

Lyndon words have many strong properties \cite{Duval1983,Lothaire1997}. Some of them have analogues in terms of Nyldon words, while some of them (many in fact) do not.  This section is dedicated to an in-depth comparison between Lyndon and Nyldon words. On the one hand, whenever the Nyldon and Lyndon words share some property, then we show that this property in the Lyndon case can also be obtained by only using their recursive definition. This is the case of Theorem~\ref{thm:Lyndon-suffix} and Proposition~\ref{prop:Lyndon-primitive}. Interestingly, these new proofs of classical properties of Lyndon words are surprisingly more complicated than their analogues for Nyldon words. This highlights that the inversion of the order is not only cosmetic, but really twists the main arguments of the proofs. On the other hand, whenever some property of the Lyndon words does not stand in the Nyldon case, then we provide explicit counter-examples.

\subsection{Sticking to the recursive definition}

The following property  of Lyndon words (Theorem~\ref{thm:Lyndon-suffix}) is well known. It is the analogue of Theorem~\ref{thm:suffix} for Nyldon words. The usual proof of this result uses the classical definition of Lyndon words: they are primitive and lexicographically minimal among their conjugates. We propose a new proof only using the recursive definition of the Lyndon words. In particular, this also provides us with another proof of Theorem~\ref{the:Lyndon-fac}, in the same vein as we proved the unicity of the Nyldon factorization. 

\begin{theorem}
\label{thm:Lyndon-suffix}
Let $w$ be a Lyndon word over $A$. For each Lyndon proper suffix $s$ of $w$, we have $s >_{\lex} w$.
\end{theorem}

\begin{proof}
Let us stress that, inside this proof, being Lyndon means belonging to the set $L$ defined in Corollary~\ref{cor:LisL}. In particular, we suppose that we know nothing else about Lyndon words than this recursive definition (such as primitivity, minimality or unicity of the Lyndon factorization).

We proceed by induction on $|w|$. If $|w|=1$, the result is obvious. If $|w|=2$, write $w={\tt ij}$ with ${\tt i,j}\in A$. The only Lyndon proper suffix of $w$ is its last letter ${\tt j}$. Since $w,{\tt i,j}$ are Lyndon, we must have ${\tt j} >_{\lex} {\tt i}$, for otherwise $({\tt i,j})$ would be a Lyndon factorization of $w$ of length at least $2$. Thus, ${\tt j} >_{\lex} w= {\tt ij}$.

Now, we suppose that $|w|\ge 3$ and that the result is true for all Lyndon words shorter than $w$. Proceed by contradiction and assume that there exists a Lyndon proper suffix $s$ of $w$ such that 
\begin{equation}
\label{eq: s<w}
	s <_{\lex} w.
\end{equation}
Note that $s\ne w$ since $s$ is a proper suffix. Among all suffixes of $w$ satisfying~\eqref{eq: s<w}, we choose $s$ to be the longest and we write $w=ps$ with $p\in A^+$. 

First, let us prove that $p$ is not Lyndon. We proceed by contradiction and we suppose that $p$ is Lyndon. Then $p<_{\lex}s$, for otherwise $(p,s)$ would be a Lyndon factorization of $w$ of length $2$. Using~\eqref{eq: s<w}, we obtain $p<_{\lex}s<_{\lex}w=ps$. Therefore $s$ must start with $p$. Write $s=ps_1$ with $s_1\ne\varepsilon$. Now, since $ps_1=s<_{\lex} w=ps$, we obtain that $s_1<_{\lex} s$. Then, since $|s|<|w|$, we obtain by induction hypothesis that $s_1$ is not Lyndon. Therefore, there exists a Lyndon factorization $(\ell_1,\ldots,\ell_k)$ of $s_1$ of length $k\ge2$. Since $s$ is Lyndon, we must have $p<_{\lex}\ell_1$, for otherwise $(p,\ell_1,\ldots,\ell_k)$ would be a Lyndon factorization of $s$ of length at least $2$. But then we have $p<_{\lex}s_1<_{\lex}s=ps_1$. Therefore, $p$ must be a prefix of $s_1$. Hence we obtain that $p^2$ is a prefix of $s$. Again, write $s_1=ps_2$ with $s_2\ne\varepsilon$. Similarly we obtain that $s_2<_{\lex}s_1$, and by using the induction hypothesis, that $s_2$ is not Lyndon. Then, similarly, $p$ must be a prefix of $s_2$, and hence $p^3$ must be a prefix of $s$. Iterating this process, we obtain that $p^n$ must be a prefix of $s$ for any $n$, which is impossible.

Now that we know that $p$ is not Lyndon, there must exist a Lyndon factorization $(p_1, \ldots, p_k)$ of $p$ of length $k\ge 2$. Then we must have $p_k <_{\lex} s$ for otherwise $(p_1, \ldots, p_k,s)$ would be a Lyndon factorization of $w$ of length at least $2$, contradicting that $w$ is Lyndon. 

Let us show that $p_ks$ cannot be Lyndon. Indeed, if $p_ks$ were Lyndon, then by definition of $s$, we would have $p_ks>_{\lex} w$. But then, since $|p_ks|<|w|$, we would get that $s>_{\lex} p_ks$ by the induction hypothesis. Gathering the last two inequalities, we would finally obtain $s>_{\lex} w$, contradicting~\eqref{eq: s<w}.

Now let $(\ell_1, \ldots, \ell_r)$ be a Lyndon factorization of $p_ks$ of length $r\ge2$. As in the proof of Theorem~\ref{thm:unicity}, we can show that there exist $i\in\{1,\ldots, r-1\}$ and $x,y \in A^{+}$ such that 
\[
	\ell_i = xy,\quad p_k=\ell_1 \cdots \ell_{i-1} x \quad\text{ and }\quad s= y \ell_{i+1} \cdots \ell_r.
\]
Now let $(y_1, \ldots, y_t)$ be a Lyndon factorization of $y$. Then $y_t$ is a Lyndon proper suffix of $\ell_i$. Since $\ell_i$ is Lyndon and $|\ell_i| < |w|$, the induction hypothesis yields $y_t >_{\lex} \ell_i$. But then $y_t >_{\lex} \ell_i \ge_{\lex} \ell_{i+1}$, hence $(y_1, \ldots, y_t,\ell_{i+1},\ldots,\ell_r)$ is a Lyndon factorization of $s$ of length at least $2$, a contradiction with the fact that $s$ is Lyndon.
\end{proof}

The converse of Theorem~\ref{thm:Lyndon-suffix} is actually true, hence we can formulate the following theorem. 

\begin{theorem}\label{thm:lex-suffix-Lyndon}
\label{thm:Lyndon-char-suffixes}
Let $w\in A^+$. Then the following assertions are equivalent.
\begin{itemize}
\item[(i)] $w$ is Lyndon.
\item[(ii)] $w$ is lexicographically smaller than all its  nonempty proper suffixes.
\item[(iii)] $w$ is lexicographically smaller than all its Lyndon proper suffixes.
\end{itemize}
\end{theorem}

\begin{proof}
It is well known that $(i)$ is equivalent to $(ii)$; see \cite{Lothaire1997} for example. Clearly $(ii)$ implies $(iii)$. We now show that $(iii)$ implies $(i)$, i.e., that the converse of Theorem~\ref{thm:Lyndon-suffix} is true. We proceed by induction on the length of the words. Since the letters are Lyndon, the base case is  trivially verified. Now suppose that $|w|\ge 2$, that $w$ is lexicographically smaller than all its Lyndon proper suffixes and that, for each word $z$ shorter than $w$, if $z$ has the property to be lexicographically smaller than all its Lyndon proper suffixes, then $z$ is Lyndon. Let us write $w=uv$ with $u,v \in A^+$. Our aim is to show that $w <_{\lex} vu$. If $v$ is Lyndon, then $w <_{\lex} v $ by hypothesis, hence $w<_{\lex} vu$. Now suppose that $v$ is not Lyndon. By applying the induction hypothesis to $v$, we obtain that there exists a Lyndon proper suffix $s$ of $v$ such that $v >_{\lex} s$. Since $s$ is also a Lyndon proper suffix of $w$, we also have $w <_{\lex} s$ by hypothesis.  Consequently, we get $w <_{\lex} s <_{\lex} v <_{\lex} vu$.
\end{proof}

Note that the proof that we give for the converse of Theorem~\ref{thm:Lyndon-suffix}, which corresponds to the implication $(iii)\implies(i)$, makes use of the usual definition of Lyndon words.
We do not know how to prove this result by using the recursive definition of Lyndon words only.

\begin{open}
Prove the implication $(iii)\implies(i)$ of Theorem~\ref{thm:Lyndon-char-suffixes} by only using the recursive definition of Lyndon words.
\end{open}

Interestingly, the converse of Theorem~\ref{thm:suffix}, which is the Nyldon analogue of Theorem~\ref{thm:Lyndon-suffix}, is {\em not} true, as illustrated in Example~\ref{ex:no-pro:lex-suffix-Nyldon}. 

\begin{example}
\label{ex:no-pro:lex-suffix-Nyldon}
Take $w={\tt 1011011}$. The Nyldon factorization of $w$ is $({\tt 101,1011})$, hence $w$ is not Nyldon. The only Nyldon proper suffixes of $w$ are ${\tt 1}$ and ${\tt 1011}$, which are both lexicographically smaller than $w$.
\end{example}

\subsection{Standard and {\v S}ir{\v s}ov factorizations}

The following very useful characterization of Lyndon words fails in the case of Nyldon words. For a proof of this result, see for example \cite{Lothaire1997}.

\begin{proposition}
A finite word $w$ over $A$ is Lyndon if and only if $w$ is a letter or there exists a factorization $(u,v)$ of $w$ into Lyndon words such that $u<_{\lex} v$. 
\end{proposition}

By Theorem~\ref{thm:standard-fac}, we know that if a  word is Nyldon, then it admits at least one factorization $(u,v)$ into Nyldon words such that $u>_{\lex} v$. However, the converse is not true as illustrated in the following example. 

\begin{example}
Take $w={\tt 1001010010}$.  Then $(u,v)=({\tt 10010100,10})$ is a factorization of $w$ into Nyldon words such that $u>_{\lex} v$.
However, $w$ is not Nyldon since its Nyldon factorization is $({\tt 10010,10010})$.
\end{example}

We have  already seen in Section~\ref{sec:standard-fac} that, similarly to the Lyndon words, the Nyldon words admit a standard factorization. Recall that the \emph{standard factorization} of a Lyndon word $w$ is the factorization $(u,v)$ of $w$ where the right factor $v$ is chosen to be the longest Lyndon proper suffix of $w$. Analogously to Theorem~\ref{thm:standard-fac}, it is well known that if $(u,v)$ is the standard factorization of a Lyndon word $w$, then the left factor $u$ is Lyndon and $u<_{\lex}v$ \cite{Lothaire1997}. The following property of the right factors of the standard factorizations of Lyndon words does not stand in the case of Nyldon words. A proof of this result can be found in \cite{Duval1983}.

\begin{proposition}
\label{prop:lex-min-suffix}
The longest Lyndon proper suffix of a Lyndon word coincides with its  lexicographically smallest nonempty proper suffix.
\end{proposition}

\begin{example}
\label{ex:lex-max-suffix}
Let $w={\tt 100}$. It is a binary Nyldon word. The longest Nyldon proper suffix of $w$ is $v={\tt 0}$. Hence the standard factorization of $w$ is $({\tt 10,0})$. However, the lexicographically greatest proper suffix of $w$ is ${\tt 00}$.
\end{example}

There is another distinguished factorization of Lyndon words:  the \emph{{\v S}ir{\v s}ov factorization} of a Lyndon word $w$ is the factorization $(u,v)$ of $w$ where the left factor $u$ is chosen to be the longest Lyndon proper prefix of $w$ \cite{Sirsov1962}. Symmetrically to the standard factorization, the {\v S}ir{\v s}ov factorization $(u,v)$ has the property that its right factor $v$ must be Lyndon and such that $u<_{\lex}v$ \cite{Viennot1978}.  The {\v S}ir{\v s}ov factorization can be seen as a left standard factorization whereas the usual standard factorization can be seen as a right standard factorization. Therefore, the Lyndon words are left and right privileged words. Nyldon words, however, are not left privileged. Indeed, most properties of Lyndon prefixes do not have analogues in terms of Nyldon words, whereas Lyndon and Nyldon words share many properties concerning their suffixes (though not all of them -- see for example Proposition~\ref{prop:lex-min-suffix} and Example~\ref{ex:lex-max-suffix}, as well as Proposition~\ref{prop:prefixes}).

In the next example, we illustrate that there cannot be any analogue to the {\v S}ir{\v s}ov factorization for Nyldon words.

\begin{example}
Consider the Nyldon word $w={\tt 10010100100}$. The longest Nyldon proper prefix of $w$ is ${\tt 100101001}$, while the corresponding suffix ${\tt 00}$ is not Nyldon. 
\end{example}

\subsection{Factorizing powers}

Lyndon words are primitive by definition. But, similarly to the proof we gave of Theorem~\ref{thm:Lyndon-suffix}, even if we suppose that all we know about Lyndon words is their recursive definition, then we can easily deduce that Lyndon word are primitive and that each primitive conjugacy class contains exactly one Lyndon word from their recursive definition. Otherwise stated, we do not need the full power of Schützenberger's theorem in order to obtain the primitivity of Lyndon words. We give a new proof of the following well-known result, where we again assume that being Lyndon only means belonging to the set $L$ defined in Corollary~\ref{cor:LisL}. Again, the point is to highlight the major differences in the behaviors of Lyndon and Nyldon words. In particular, the proof shows that the Lyndon factorizations of powers are straightforwardly obtained from the Lyndon conjugate of their primitive roots. In view of Example~\ref{ex:Nyldon-powers} below, we see that the same reasoning does not work for Nyldon words.

\begin{proposition}
\label{prop:Lyndon-primitive}
Lyndon words are primitive and each primitive conjugacy class contains exactly one Lyndon word.
\end{proposition}

\begin{proof}
We proceed by induction on the length $n$ of the words. Since all letters are Lyndon, the base case $n=1$ is verified. Now suppose that $n>1$ and that all Lyndon words of length less than $n$ are primitive and that there is exactly one element in each primitive conjugacy class which is Lyndon. Let $w$ be a power of length $n$. Then $w=x^m$, for some primitive word $x$ and $m\ge 2$. By induction hypothesis, we know that $x$ possesses a Lyndon conjugate: $y=vu$ is Lyndon and $x=uv$. Then $w=x^m=(uv)^m=u(vu)^{m-1}v=u y^{m-1} v$. Let $(u_1,\ldots, u_k)$ and $(v_1,\ldots, v_\ell)$ be the Lyndon factorizations of $u$ and $v$ respectively. Then $u_k\ge_{\lex} y\ge_{\lex} v_1$ by Theorem~\ref{thm:Lyndon-suffix}. Therefore $w=u_1\cdots u_k \cdot y^{m-1} \cdot  v_1\cdots v_\ell$ is not Lyndon. So far, we have obtained that all Lyndon words of length $n$ are primitive. 

Now suppose that there exist distinct Lyndon words $x,y$ of length $n$ in the same conjugacy class. Let $x=uv$ and $y=vu$ with $u,v\neq\varepsilon$. Then, in the same way as in the previous paragraph, we obtain that  $x^2$ has two distinct Lyndon factorizations: $(x,x)$ and $(u_1,\ldots, u_k, y, v_1,\ldots, v _\ell)$. But Theorem~\ref{thm:Lyndon-suffix} implies the unicity of the Lyndon factorization (similarly as Theorem~\ref{thm:suffix} implies the unicity of the Nyldon factorization), hence we have reached a contradiction.
\end{proof}

The fact that all Nyldon are primitive is a consequence of Schützenberger's theorem; see Theorems~\ref{thm:Schutzenberger65} and~\ref{thm:conjugacy}. It would thus be interesting to be able to understand the Nyldon factorizations of powers. Indeed, if we could effectively deduce the Nyldon factorization of a power $w=u^k$ from the Nyldon conjugate of its primitive root $u$, then we would obtain a much simpler proof of the fact that no power can be Nyldon. 
In the following example, we show some surprising Nyldon factorizations of successive powers of some primitive word.

\begin{example}
\label{ex:Nyldon-powers}
The primitive word 
\[
	u={\tt 0111101 1011111011110111}
\] 
is not Nyldon since its  Nyldon factorization is ${\tt (0,1,1,1,101,1011111011110111)}$. The Nyldon conjugate of $u$ is 
\[
	n={\tt 10111101101111101111011}={\tt 1}u{\tt 1}^{-1}.
\] 
We write $u=ps$, with $p={\tt 0111101}$ and $s={\tt 1011111011110111}$. Thus, the Nyldon factorization of 
$u$ is given by ${\tt (0,1,1,1,101},s)$. Next, the Nyldon factorization of $u^2$ is given by ${\tt (0,1,1,1,101},x,yp,s)$, where $x={\tt 101111}$ and $xy=s$. Table~\ref{table:Nyldon-powers} stores the Nyldon factorizations of the powers of $u$.
\begin{table}[h]
{\centering
\[
	\begin{array}{c | l }
	k 		& \text{Nyldon factorization of } u^k 											\\
	\hline 
	1 		& ({\tt 0,1,1,1,101},s)														\\
	2 		& ({\tt 0,1,1,1,101},x,yp,s)								\\	
	3 		& ({\tt 0,1,1,1,101},x,y,px,yp,s)		\\
	\ge 4 	 & ({\tt 0,1,1,1,101},x,y{\tt 1}^{-1},n^{k-4},n{\tt 1}px,yp,s)	
	\end{array}
\]}
\caption{The Nyldon factorizations of the powers of $u={\tt 01111011011111011110111}$.}
\label{table:Nyldon-powers}
\end{table}
\end{example}

In fact, it seems surprisingly difficult to understand the Nyldon factorizations of powers. We leave this concern for future work and state the following open problem.

\begin{open}
Given a primitive word $u$, is it true that there exists some positive integer $K$ such that, for all $k\ge K$, the Nyldon factorization of $u^k$ is of the form $(p_1,\ldots,p_m,v^{k-K},s_1,\ldots,s_n)$ where $v$ is the Nyldon conjugate of $u$? If yes, characterize the smallest such $k$. More generally, describe the Nyldon factorizations of powers in terms of the Nyldon conjugates of their primitive roots. \end{open}

\subsection{About codes}

We end our comparison between Nyldon and Lyndon words by a discussion on circular codes and comma-free codes. Recall that a subset $F$ of $A^*$ is a \emph{code} if for any $x_1,\ldots,x_m,y_1,\ldots,y_n$ in $F$, we have $x_1 \cdots x_m = y_1 \cdots y_n$ if only if $m=n$ and $x_i =y_i$ for all $i\in\{1,\ldots,m\}$. 

\begin{definition}
Let $F$ be a code. 
\begin{itemize}
\item $F$ is said to be a {\em circular code} if for any $u,v\in A^*$, we have $uv,vu\in F^* \implies u,v\in F^*$.
\item  $F$ is said to be a \emph{comma-free code} if for any $w\in F^+$ and $u,v\in A^*$, we have $uwv\in F^*\implies u,v\in F^*$.
\end{itemize}
\end{definition}

The Lyndon words of length $n$ over a $k$-letter alphabet that form a comma-free code are completely characterized in terms of $n$ and $k$, see for example \cite{Berstel-Perrin-Reutenauer2010}. 

\begin{theorem}
Let $A$ be an alphabet of size $k$ and let $n\ge 1$. Then the set $\Lyn\cap A^n$ of Lyndon words of length $n$ over $A$ is a comma-free code if and only if $n=1$, or $n=2$ and $k\in\{2, 3\}$, or $n\in\{3,4\}$ and $k= 2$. 
\end{theorem}

We provide an analogous result for Nyldon words. In particular, note that the Nyldon and Lyndon words of length $n$ over some alphabet $A$ do not necessarily form a comma-free code simultaneously. Surprisingly enough, Nyldon words more often form a comma-free code than Lyndon words. We cut the proof of Theorem~\ref{the:comma-free} in several technical lemmas.

\begin{lemma}
\label{lem:comma-free1}
The set $\Nyl\cap A$ is a comma-free code for any alphabet $A$.
\end{lemma}

\begin{proof}
This is immediate since $\Nyl\cap A=A$.
\end{proof}

\begin{lemma}
Let $A$ be an alphabet of size $k$. Then $\Nyl\cap A^2$ is a comma-free code if and only if $k\in\{2,3\}$.
\end{lemma}

\begin{proof}
If $k=2$, then $A=\{{\tt 0,1}\}$ and $\Nyl\cap A^2=\{{\tt 10}\}$. In this case, $\Nyl\cap A^2$ is clearly a comma-free code. 

Suppose that $k=3$. Then $A=\{{\tt 0,1,2}\}$ and $\Nyl\cap A^2=\{{\tt 10, 20, 21}\}$. Take $x\in (\Nyl\cap A^2)^+$ and $u,v\in A^*$ such that $uxv \in (\Nyl\cap A^2)^*$. Then there exists $\ell \ge 1$ and words $y_1, \ldots, y_\ell$ in $\Nyl\cap A^2$ such that $uxv = y_1 \ldots y_\ell$. Since words in $\Nyl\cap A^2$ are all of the same length $2$, in order to show that $u,v\in  (\Nyl\cap A^2)^*$, it is enough to prove that $x$ cannot start strictly within a factor $y_i$. Suppose to the contrary that it does. Then $i\in\{1,\ldots,\ell-1\}$. If $y_i\in\{{\tt 10, 20}\}$, then $x$ must start with ${\tt 0}$, whence $x\notin (\Nyl\cap A^2)^+$. Thus $y_i={\tt 21}$. Since $x\in (\Nyl\cap A^2)^+$, this implies that $x$ starts with ${\tt 10}$ and $y_{i+1}$ starts with ${\tt 0}$, which is impossible. This proves that $\Nyl\cap A^2$ is a comma-free code in this case as well.

Now, suppose that $k\ge 4$. Then $\{{\tt 0,1,2,3}\}\subseteq A$ and the words ${\tt 10, 21, 32}$ belong to $\Nyl\cap A^2$. Take $x={\tt 21}$, $u={\tt 3}$ and $v={\tt 0}$. Then $uxv = {\tt 3(21)0} = {\tt (32)(10)} \in N_2^*$ but $u,v \notin  (\Nyl\cap A^2)^*$. Therefore, $\Nyl\cap A^2$ is not a comma-free code if $k\ge4$.
\end{proof}

\begin{lemma}
Let $A$ be an alphabet of size $k$ and $n\in\{3,4,5,6\}$. Then $\Nyl\cap A^n$ is a comma-free code if and only if $k=2$.
\end{lemma}

\begin{proof}
First, assume that $k=2$. Then $A=\{{\tt 0,1}\}$ and it is easily verified that $\Nyl\cap A^n$ are comma-free codes for $n\in\{3,4,5,6\}$. Let us give some more details for the case $n=5$. We have $\Nyl\cap A^5=\{{\tt 10000, 10001, 10010, 10011 ,10110,10111}\}$. Take $x\in (\Nyl\cap A^5)^+$ and $u,v\in A^*$, and assume that $uxv \in (\Nyl\cap A^5)^*$. Then there exists $\ell \ge 1$ and words $y_1, \ldots, y_\ell$ in $\Nyl\cap A^5$ such that $uxv = y_1 \cdots y_\ell$. Since words in $\Nyl\cap A^5$ are all of the same length $5$, in order to show that $u,v\in  (\Nyl\cap A^5)^*$, it is enough to prove that $x$ cannot start strictly within a factor $y_i$. It is easily seen to be the case
since all binary Nyldon words of length greater than or equal to $2$ start with ${\tt 10}$ and no Nyldon words start with ${\tt 1010}$.
This proves that $\Nyl\cap A^5$ is a comma-free code. The cases $n\in\{3,4,6\}$ are similar.

Second, we show that $\Nyl\cap A^n$ is not a comma-free code whenever $k\ge 3$ and $n\in\{3,4,5,6\}$. Let $k\ge 3$. Then $\{{\tt 0,1,2}\}\subseteq A$. We start with the case $n=3$. Since the words ${\tt 101}$ and ${\tt 210}$ both belong to $\Nyl\cap A^3$ and since ${\tt 2(101)01} = {\tt (210)(101)}$, we get that $\Nyl\cap A^3$ is not a comma-free code. Similarly, since 
\begin{align*}
&{\tt 1000, 1001, 2100}\in\Nyl\cap A^4 \text{ and } {\tt 2(1001)000 = (2100)(1000)}, \\
&{\tt 10000, 10001, 21000}\in\Nyl\cap A^5 \text{ and }  {\tt 2(10001)0000 = (21000)(10000)}, \\
&{\tt 100000, 100001, 210000}\in \Nyl\cap A^6  \text{ and } {\tt 2(100001)00000 = (210000)(100000)},
\end{align*}
we get that the codes $\Nyl\cap A^n$ are not comma-free for $n\in\{3,4,5\}$. 
\end{proof}

\begin{lemma}
\label{lem:comma-free2}
If $n\ge 7$, then $\Nyl\cap A^n$ is not a comma-free code for any alphabet $A$ of size greater than or equal to $2$.
\end{lemma}

\begin{proof}
Suppose that $n\ge 7$ and $A$ is an alphabet containing $\{{\tt 0,1}\}$. Then, clearly, the words ${\tt 10110}^{n-4}, {\tt 1001010}^{n-6}$ and ${\tt 10}^{n-4}{\tt 100}$ belong to $\Nyl\cap A^n$. Since ${\tt 101 (10}^{n-4}{\tt 100) 1010}^{n-6} = ({\tt 101 10}^{n-4}) ({\tt 100 1010}^{n-6})$, we obtain that $\Nyl\cap A^n$ is not a comma-free code. 
\end{proof}

Thanks to Lemmas~\ref{lem:comma-free1} to~\ref{lem:comma-free2}, we obtain a characterization of the sets of Nyldon words of length $n$ that form a comma-free code.

\begin{theorem}
\label{the:comma-free}
Let $A$ be an alphabet of size $k$ and let $n\ge 1$. Then the set $\mathcal{N} \cap A^n$ of Nyldon words of length $n$ over $A$ is a comma-free code if and only if $n=1$, or $n=2$ and $k\in\{2, 3\}$, or $n\in\{3,4,5,6\}$ and $k= 2$. 
\end{theorem}

Clearly, comma-free codes are always circular codes. Further, circular codes made up with words of the same length $n$ only contain primitive words and at most one element of each primitive conjugacy class of length $n$ \cite{Berstel-Perrin-Reutenauer2010}. Lyndon words of any fixed length $n$ are known to form a circular code, see for example \cite{Berstel-Perrin-Reutenauer2010}. However, we do not know whether this property holds true in the case of Nyldon words.

\begin{open}
Do the Nyldon words of any fixed length $n$ form a circular code?
\end{open}

\section{Lazard factorizations}
\label{sec:Lazard}

A well-known class of complete factorizations of the free monoid is that of Lazard sets~\cite{Viennot1978}. For our purpose, we need to make a distinction between left Lazard sets and right Lazard sets. In~\cite{Viennot1978}, a Lazard factorization corresponds to what is called here a left Lazard set.

\subsection{Background}

\begin{definition}
\label{def:Lazard}
A {\em left} (resp., {\em right}) {\em Lazard set} is a subset $F$ of $A^+$ endowed with some total order $\prec$ and such that, for any integer $n\ge 1$, if $F\cap A^{\le n}=\{u_1,\ldots,u_k\}$ with $k\ge1$ and $u_1\prec \cdots\prec u_k$ (resp., $u_1\succ \cdots\succ u_k$), and if we consider the sequence $(Y_i)_{i\ge 1}$ of sets defined by $Y_1=A$ and, for $i\ge1$, $Y_{i+1}=u_i^*(Y_i\setminus \{u_i\})$  (resp., $Y_{i+1}=(Y_i\setminus \{u_i\})u_i^*$), then we have
		\begin{itemize}
		\item[(i)] for every $i\in\{1,\ldots,k\}$, $u_i\in Y_i$
		\item[(ii)] $Y_k\cap A^{\le n}=\{u_k\}$.
		\end{itemize}	
A both left and right Lazard set is called a \emph{Viennot set}.
\end{definition}

Observe that the sets $Y_i$ of the left (resp., right) Lazard construction are all prefix-free (resp., suffix-free): no word in $Y_i$ is the prefix (resp., suffix) of another word in $Y_i$. Originally, Viennot sets were called {\em regular factorizations} \cite{Viennot1978}. The Lyndon words form a Viennot set if $\prec$ is chosen to be $<_{\lex}$. In Examples~\ref{ex:Lyndon-left-Lazard} and~\ref{ex:Lyndon-right-Lazard}, we illustrate the left and right Lazard constructions of the Lyndon words. For other examples of Lazard sets, see, for instance~\cite{Viennot1978} and, more recently,~\cite{Perrin-Reutenauer2018}. 

\begin{example}
\label{ex:Lyndon-left-Lazard}
We illustrate the fact that the binary Lyndon words form a left Lazard set. We compute the binary Lyndon words up to length $5$ thanks to the left Lazard construction of Definition~\ref{def:Lazard}. Let $A_1=\{{\tt 0,1}\}$ and, for $i\ge1$, define $A_{i+1} = a_i^* (A_i\setminus\{a_i\}) $ with $a_i = \min_{<_{\lex}} (A_i  \cap \{{\tt 0,1}\}^{\le 5})$. We see in Table~\ref{table:Lyndon-Lazard} that the procedure from Definition~\ref{def:Lazard} halts after 14 steps. The binary Lyndon words of length at most 5 are exactly $a_1, \ldots , a_{14}$. Note that  the left Lazard procedure yields the Lyndon words in the increasing lexicographic order.
\begin{table}[h] 
\centering
\[
	\begin{array}{c || l | l|| l | l}
	i 	& A_i \cap \{{\tt0,1}\}^{\le 5} 		& a_i 		& B_i \cap \{0,1\}^{\le 5} 			& b_i 		\\
	\hline 
	1 	& \{{\tt 0,1}\} 						& {\tt 0} 		& \{{\tt 0,1}\} 						& {\tt 1} 		\\
	2 	& \{{\tt 1, 01, 001, 0001, 00001}\} 	& {\tt 00001} 	& \{{\tt 0, 01, 011, 0111, 01111}\} 	& {\tt 01111} 	\\	
	3 	& \{{\tt 1, 01, 001, 0001}\} 			& {\tt 0001} 	& \{{\tt 0, 01, 011, 0111}\} 			& {\tt 0111} 	\\
	4 	& \{{\tt 1, 00011, 01, 001}\}			& {\tt 00011} 	& \{{\tt 0, 00111, 01, 011}\}			& {\tt 011} 	\\
	5 	& \{{\tt 1, 01, 001}\} 				& {\tt 001} 	& \{{\tt 0, 0011, 00111, 01, 01011}\} 	& {\tt 01011}	\\
	6 	& \{{\tt 1, 0011, 01, 00101}\} 		& {\tt 00101} 	& \{{\tt 0, 0011, 00111, 01}\} 		& {\tt 01} 	\\
	7 	& \{{\tt 1, 0011, 01}\}  				& {\tt 0011} 	& \{{\tt 0, 001, 00101, 0011, 00111}\} 	& {\tt 00111}	\\
	8 	& \{{\tt 1, 00111, 01}\} 				& {\tt 00111}	& \{{\tt 0, 001, 00101, 0011}\} 		& {\tt 0011} 	\\
	9 	& \{{\tt 1, 01}\}					& {\tt 01}		& \{{\tt 0, 00011, 001, 00101}\} 		& {\tt 00101} 	\\
	10 	& \{{\tt 1, 011, 01011}\} 			& {\tt 01011} 	& \{{\tt 0, 00011, 001} \} 			& {\tt 001} 	\\
	11 	& \{{\tt 1, 011}\} 					& {\tt 011} 	& \{{\tt 0, 0001, 00011}\} 			& {\tt 00011}	\\
	12 	& \{{\tt 1, 0111}\} 					& {\tt 0111} 	& \{{\tt 0, 0001}\} 					& {\tt 0001}	\\
	13 	& \{{\tt 1, 01111}\} 					& {\tt 01111} 	& \{{\tt 0, 00001}\} 					& {\tt 00001} 	\\
	14 	& \{{\tt 1}\} 						& {\tt 1} 		& \{{\tt 0}\} 						& {\tt 0}
	\end{array}
\]
\caption{The Lyndon words form a Viennot set.}
\label{table:Lyndon-Lazard}
\end{table}
\end{example}

\begin{example}
\label{ex:Lyndon-right-Lazard}
Now let us illustrate the fact that the binary Lyndon words also form a right Lazard set. We compute the binary Lyndon words up to length $5$ thanks to the right Lazard construction of Definition~\ref{def:Lazard}. Let $B_1=\{{\tt 0,1}\}$ and, for $i\ge1$, define $B_{i+1} =  (B_i\setminus\{b_i\})b_i^* $ with $b_i = \max_{<_{\lex}} (B_i  \cap \{{\tt 0,1}\}^{\le 5})$. As before, the procedure from Definition~\ref{def:Lazard} ends after 14 steps; see Table~\ref{table:Lyndon-Lazard}. The words $b_1, \ldots, b_{14}$ are exactly the Lyndon words of length at most 5. Observe that, this time, the right Lazard procedure yields the Lyndon words in the decreasing lexicographic order.
\end{example}

\begin{remark}
\label{rem:order}
In view of Definition~\ref{def:Lazard}, there are two ways of thinking of the status of the order of the elements of $F$. First, we can think of the order on $F$ to be induced by some preexisting order on $A^*$. In this case, we first fix some total order $\prec$ on $A^ *$, and then, there are two possibilities: either there is a corresponding (left and/or right) Lazard set, or there is not. Otherwise stated, either the Lazard procedure ends for each length $n\ge 1$, meaning that, for each $n\ge 1$, there exists some $k\ge1$ such that (i) and (ii) hold. In this case, for each length $n$, the left (resp., right) Lazard procedure computes the words of length $n$ in $F$ one by one by outputting the least (resp., greatest) words in $Y_i\cap A^{\le n}$ (with respect to the preexisting order $\prec$ on $A^*$) until it reaches a singleton set $Y_k\cap A^{\le n}$. This is what we did for the Lyndon words in Examples~\ref{ex:Lyndon-left-Lazard} and~\ref{ex:Lyndon-right-Lazard}. We first considered the (increasing) lexicographic order on $\{{\tt 0,1}\}^*$, and then, at each step of the procedure, we outputted the lexicographically least (resp., greatest) word in $A_i\cap\{{\tt 0,1}\}^n$ (resp., $B_i\cap\{{\tt 0,1}\}^n$). 

In fact, we do not need to have a preexisting total order on $A^*$ at our disposal to define a Lazard set $F$. Instead, we can think of the total order $\prec$ on $F$ to be induced by the Lazard process itself. In this case, the choice of the words $u_i$ that are removed from the sets $Y_i$ at each step is what determines the total order on $F$: the fact that $u_i$ is outputted before $u_{i+1}$ implies that $u_i\prec u_{i+1}$ (resp., $u_i\succ u_{i+1}$) for a left (resp., right) Lazard process. In particular, since we always have $A=Y_1\subset F$, this process always induces a total order on the alphabet $A$. However, there is no reason that the total order $\prec$ on $F$ naturally extends to a total order on the free monoid $A^*$. 
\end{remark}

In view of the following result, Lazard sets are sometimes also called {\em Lazard factorizations}.

\begin{theorem}\cite{Viennot1978}
\label{thm:Lazard-implies-complete}
All left (resp., right) Lazard sets with respect to some total order $\prec$ are complete factorizations of the free monoid with respect to the same total order $\prec$.
\end{theorem}

However, it is not true that all complete factorizations of the free monoid can be obtained by the Lazard procedure. The following result of Viennot characterizes the Lazard sets among the complete factorizations of the free monoid.

\begin{theorem}\cite{Viennot1978}
\label{thm:Viennot}
Let $F$ be a complete factorization of $A^{*}$ with respect to some total order $\prec$. Then $F$ is a left (resp., right) Lazard set with respect to $\prec$ if and only if for all $f,g \in F$, $fg \in F$ implies $f \prec fg$ (resp.,  $g \succ fg$).
\end{theorem}

From Theorem~\ref{thm:Viennot}, one can deduce a characterization of Viennot sets. 

\begin{corollary} 
\label{cor:Viennot-fact}
Let $F$ be a complete factorization of the free monoid $A^{*}$ with respect to some total order $\prec$. Then $F$ is a Viennot set if and only if, for all $f,g \in F$, $fg \in F$ implies $f \prec fg \prec g$.
\end{corollary} 

\begin{example}
Thanks to Viennot's characterization, we see once again that Lyndon words are a Viennot set with respect to the (increasing)  lexicographic order $<_{\lex}$. Indeed, for all Lyndon words $f$ and $g$, if $fg$ is also Lyndon, then $f <_{\lex} fg <_{\lex} g$ by Theorem~\ref{thm:lex-suffix-Lyndon}. Observe that the first inequality directly follows from the definition of the lexicographic order, and is thus  verified by any words $f$ and $g$. As for the second inequality, it is valid for Lyndon words only. 
\end{example}

\subsection{Nyldon words form a right Lazard factorization}

In the following, we show that the Nyldon words form a right Lazard factorization, but not a left one. Consequently, the set of Nyldon words is not Viennot.

\begin{proposition}
\label{prop:Nyldon-not-left-Lazard}
The set $\Nyl$ of Nyldon words is not a left Lazard factorization. 
\end{proposition}

\begin{proof}[First proof of Proposition~\ref{prop:Nyldon-not-left-Lazard}]
Proceed by contradiction and suppose that $\Nyl$ is a left Lazard factorization associated with some order $\prec$. In view of Theorems~\ref{thm:unicity} and~\ref{thm:Lazard-implies-complete}, the order $\prec$ on $\Nyl$ must coincide with the decreasing lexicographic order $>_{\lex}$. Thanks to   Theorem~\ref{thm:Viennot}, for all $f,g \in \Nyl$ such that $fg \in \Nyl$, we should have $f >_{\lex} fg$, a contradiction.
\end{proof}

Let us now give another proof of Proposition~\ref{prop:Nyldon-not-left-Lazard}, which is not based on Viennot's characterization.

\begin{proof}[Second proof of Proposition~\ref{prop:Nyldon-not-left-Lazard}]
If $\Nyl$ were a left Lazard factorization associated with some order $\prec$, then, similarly to the previous proof, we know that the only possible choice for the order $\prec$ on $\Nyl$ would be $>_{\lex}$. Let us show that the set $\Nyl$ cannot be obtained thanks to the left Lazard procedure with respect to  $>_{\lex}$. Without loss of generality, we suppose that $A=\{{\tt 0,1,\ldots,{\tt m}}\}$, with ${\tt 0<_{\lex} 1<_{\lex}\cdots <_{\lex}  \tt m}$. Suppose instead that we can produce the Nyldon words in the decreasing lexicographic order by applying the left Lazard procedure. Then, with the notation of Definition~\ref{def:Lazard}, for any length $n$, we must have $Y_1= A$ and $u_1=\min_{>_{\lex}} Y_1=\max_{<_{\lex}} Y_1= {\tt m}$. Then $Y_2 =  {\tt m}^* \{{\tt 0,1},\ldots,{\tt m-1}\}$. If $n\ge 3$, then the word ${\tt mm0}$ must eventually be outputted by the procedure. But ${\tt mm0}$ is not Nyldon by Proposition~\ref{prop:prefixes}. We have thus reached a contradiction.
\end{proof}

Next, our aim is to show that the Nyldon words form a right Lazard set. This  essentially follows from Theorem~\ref{thm:unicity} and Theorem~\ref{thm:Viennot}. 

\begin{theorem}\label{thm:Nyldon-right-Lazard}
The Nyldon words equipped with the decreasing lexicographic order $>_{\lex}$ form a right Lazard set.
\end{theorem}

\begin{proof}
On the one hand, we know from Theorem~\ref{thm:unicity} that the Nyldon words form a complete factorization of the free monoid $A^{*}$ with respect to the decreasing lexicographic order $>_{\lex}$. On the other hand, Theorem~\ref{thm:suffix} tells us that,  if $w$ is a Nyldon word and if $s$ is a Nyldon proper suffix of $w$, then $w>_{\lex} s$. In particular, if $f$ and $g$ are Nyldon words such that $fg$ is also Nyldon, then we have $g<_{\lex}fg$. Thus, we obtain from Theorem~\ref{thm:Viennot} that the Nyldon words ordered by the decreasing lexicographic order $>_{\lex}$ form a right Lazard set.
\end{proof} 

\begin{example}
\label{ex:Nyldon-right-Lazard}
We illustrate that the binary Nyldon words can be obtained as a right Lazard set. Table~\ref{table:Nyldon-right-Lazard} shows the sets $Y_i$ and the words $u_i$ of Definition~\ref{def:Lazard} corresponding to $n=5$. 
Note that the Nyldon words are produced in increasing lexicographic order by the right Lazard procedure, since by definition, they are produced in decreasing order with respect to the decreasing lexicographic order $>_{\lex}$.
Thus, at each step we must have  $u_i=\max_{>_{\lex}} (Y_i\cap \{{\tt 0,1}\}^{\le 5})=\min_{<_{\lex}} (Y_i\cap \{{\tt 0,1}\}^{\le 5})$.
\begin{table}[h]
\centering
\[
	\begin{array}{c || l | l}
	i 	& Y_i \cap \{{\tt 0,1}\}^{\le 5} 												& u_i \\
	\hline 
	1 	& \{{\tt 0,1}\} 																& {\tt 0} \\
	2 	& \{{\tt 1, 10, 100, 1000, 10000}\} 											& {\tt 1} \\
	3 	& \{{\tt 10, 101, 1011, 10111, 100, 1001, 10011, 1000, 10001, 10000}\}  				& {\tt 10} \\
	4 	& \{{\tt 101, 10110, 1011, 10111, 100, 10010, 1001, 10011, 1000, 10001, 10000}\}  	& {\tt 100}\\
	5 	& \{{\tt 101, 10110, 1011, 10111, 10010, 1001, 10011, 1000, 10001, 10000}\} 		& {\tt 1000}\\
	6 	& \{{\tt 101, 10110, 1011, 10111, 10010, 1001, 10011, 10001, 10000}\} 				& {\tt 10000}\\
	7 	& \{{\tt 101, 10110, 1011, 10111, 10010, 1001, 10011, 10001}\}  					& {\tt 10001}\\
	8 	& \{{\tt 101, 10110, 1011, 10111, 10010, 1001, 10011}\} 							& {\tt 1001}\\
	9 	& \{{\tt 101, 10110, 1011, 10111, 10010, 10011}\} 								& {\tt 10010} \\
	10 	& \{{\tt 101, 10110, 1011, 10111, 10011}\}  										& {\tt 10011}\\
	11 	& \{{\tt 101, 10110, 1011, 10111}\}  											& {\tt 101} \\
	12 	& \{{\tt 10110, 1011, 10111}\} 						 						& {\tt 1011} \\
	13 	& \{{\tt 10110, 10111}\} 				 									& {\tt 10110}\\
	14 	& \{{\tt 10111}\} 															& {\tt 10111}\\
	\end{array}
\]
\caption{The Nyldon words form a right Lazard factorization.}
\label{table:Nyldon-right-Lazard}
\end{table}
\end{example}

\begin{remark}
We used the unicity of the Nyldon factorization (that is, Theorem~\ref{thm:unicity}) to obtain that the set of Nyldon words is a right Lazard set. Indeed, in order to be able to apply Viennot's characterization (that is, Theorem~\ref{thm:Viennot}), we need to first know that the set of Nyldon words form a complete factorization of the free monoid.
\end{remark}

As is well known, with any (right or left) Lazard set is associated a basis of the free Lie algebra \cite{Viennot1978,Reutenauer1993}. As a consequence of Theorem~\ref{thm:Nyldon-right-Lazard}, Nyldon words can be used in order to obtain a new basis of the free Lie algebra.

It is interesting to compare the effectiveness of the Lazard procedure for computing the Lyndon words and the Nyldon words. In fact, it seems that the Nyldon words are produced much faster by the Lazard procedure than the Lyndon words.
For example, in order to compute the $14$ Lyndon words up to length $5$, both the left and right Lazard procedures actually need $13$ steps whereas all $14$ Nyldon words of length at most $5$ are already computed at the fourth step. 
Indeed, on the one hand, in Table~\ref{table:Nyldon-right-Lazard}, we see that $\Nyl\cap\{{\tt 0,1}\}^{\le 5}\subset Y_4$. 
On the other hand, we see in Table~\ref{table:Lyndon-Lazard} that the word $a_{13}={\tt 01111}$ belongs to $A_{13}\setminus A_{12}$ and, similarly, the word $b_{13}=00001$ belongs to $B_{13}\setminus B_{12}$. 

In general, if $k=\#(\Lyn\cap A^{\le n})$, then both the left and right Lazard procedures need $k-1$ steps in order to compute all the Lyndon words of length up to $n$. More precisely, the penultimate outputted words by the left and right Lazard procedures are such that $a_{k-1}\in A_{k-1} \setminus \cup_{i=1}^{k-2}A_i$ and $b_{k-1}\in B_{k-1} \setminus \cup_{i=1}^{k-2}B_i$ respectively. Let us give more details in the case of the left Lazard procedure, the right case being symmetric. Since the left Lazard procedure outputs the Lyndon words in increasing lexicographic order, we know that $a_{k-2}={\tt 01}^{n-2}\in A_{k-2}$ and $a_{k-1}={\tt 01}^{n-1}\in A_{k-1}$. Now, because the sets $A_i$ are prefix-free and $a_{k-2}$ is a prefix of $a_{k-1}$, we obtain that $a_{k-1}\notin\cup_{i=1}^{k-2}A_i$. However, for a given $n$, the sets $A_i$ and $B_i$ are in general not of the same size; see, for instance, Table~\ref{table:Lyndon-Lazard}. 

We leave it as an open problem to characterize the number of steps actually needed by the right Lazard procedure in order to produce all Nyldon words up to length $n$.

\begin{open}
Let $n\ge 1$ and let $(Y_i)_{i\ge 1}$ the sequence of sets defined by $Y_1=A$ and, for $i\ge1$, $Y_{i+1}=(Y_i\setminus \{u_i\})u_i^*$, where $\Nyl\cap A^{\le n}=\{u_1,\ldots,u_k\}$ with $u_1<_{\lex}\cdots <_{\lex} u_k$. Characterize the least integer $j\ge 1$ such that $Y_j\cap A^{\le n}$ contains all Nyldon words over $A$ up to length $n$.
\end{open}

\begin{remark}
For each Lazard set $F\subseteq A^+$ associated with some total order $\prec$, there exists a total order $<_F$ on $A^*$ extending $\prec$ and such that an arbitrary word $w$ over $A$ belongs to $F$ if and only if, for all factorizations $(u,v)$ of $w$ into nonempty words, we have $w<_F vu$ \cite{Melancon1992}. In other words, the words in a Lazard set $F$ are primitive and minimal among their conjugates with respect to $<_F$. Words in $F$ are also characterized by the fact of being smaller than all their  nonempty proper suffixes with respect to $<_F$. We refer to \cite{Melancon1992} for the formal definition of the order $<_F$.
As it happens, in the case of Lyndon words, the order $<_{\Lyn}$ coincide with the lexicographic order $<_{\lex}$ on $A^*$ \cite{Melancon1992}. However, even though the restriction of $<_{\Nyl}$ to $\Nyl$ coincide with $>_{\lex}$ (as is the case for every Lazard set), the order $<_{\Nyl}$ and $>_{\lex}$ differs on $A^*\setminus \Nyl$ since Nyldon words are not lexicographically maximal among their conjugates. 
\end{remark}

\subsection{Lexicographically extremal choices in the Lazard procedure}

We saw in Example~\ref{ex:Lyndon-left-Lazard} that the choice of the lexicographically minimal elements of the sets $Y_i$ at each step of the left Lazard procedure leads to the set of Lyndon words. Symmetrically, the choice of the lexicographically maximal elements of the sets $Y_i$ at each step of the right Lazard procedure also yields the set of Lyndon words; see Example~\ref{ex:Lyndon-right-Lazard}. In the case of the Nyldon words, we have seen in Example~\ref{ex:Nyldon-right-Lazard} that by choosing the  lexicographically minimal elements of the sets $Y_i$ at each step of the right Lazard procedure produces the Nyldon words. Therefore, it is natural to ask what happens when we choose the lexicographically maximal elements of the sets $Y_i$ at each step of the left Lazard procedure. As it happens, this procedure also leads to the set of Lyndon words. More precisely, we obtain the set $\overline{\Lyn}=\{\overline{w}\in \{{\tt 0,1,\ldots,m}\}^*\colon w\in\Lyn \}$ where $\overline{\tt i}={\tt m{-}i}$ and $\overline{uv}=\overline{u}\,\overline{v}$ for all finite words $u,v$ over $A=\{{\tt 0,1,\ldots,m}\}$. Moreover, we obtain the  words of this set in the increasing lexicographic order induced by the total order ${\tt 0>1>\cdots>m}$ on the letters (also called the {\em inverse lexicographic order} \cite{Felice2018}). In fact, $\overline{\Lyn}$ is also a right Lazard set, as shown by Proposition~\ref{prop:permutation-letters-Lazard} below.

\begin{definition}
\label{def:pi}
Suppose that $A=\{{\tt 0,1,\ldots,m}\}$ and that $\pi$ is a permutation acting on $A$. If $\prec$ is a total order on $A^*$ which is induced by the total order ${\tt 0<1<\cdots<m}$ on the letters, then we denote by $\prec^{\pi}$ the corresponding total order which is induced by the total order ${\tt \pi(0)<\pi(1)<\cdots<\pi(m)}$ on the letters. Moreover, we extend the definition of $\pi$ on $A^+$ by setting $\pi(uv)=\pi(u)\pi(v)$ for all $u,v\in A^+$. Then, as usual, if $F$ is a subset of $A^+$ then $\pi(F)$ designates the set $\{\pi(w)\colon w\in F\}$.
\end{definition}

\begin{example}
If $\pi$ is the identity, then the order $\prec^\pi$ coincide with the original order $\prec$.
\end{example}

\begin{example}
\label{ex:Lyn-bar}
With the notation of Definition~\ref{def:pi}, the set $\overline{\Lyn}$ described above corresponds to $\pi(\Lyn)$ where $\pi$ is the permutation on $\{{\tt 0,1,\ldots,m}\}$ defined by $\pi({\tt i})={\tt m{-}i}$ for all ${\tt i}$. Thus, Proposition~\ref{prop:permutation-letters-Lazard} below shows in particular that $\overline{\Lyn}$ is a Viennot set ordered by the increasing lexicographic order $<_{\lex}^\pi$ induced by ${\tt 0>1>\cdots>m}$.
\end{example}

\begin{proposition}
\label{prop:permutation-letters-Lazard}
Suppose that $A=\{{\tt 0,1,\ldots,m}\}$ and that $\pi$ is a permutation acting on $A$. If a subset $F$ of $A^+$ is a left (resp., right) Lazard set  with respect to a total order order $\prec$ which is induced by the total order ${\tt 0<1<\cdots<m}$ on the letters, then the set $\pi(F)$ equipped with the total order $\prec^{\pi}$ is also a left (resp., right) Lazard set.
\end{proposition}

\begin{proof}
We do the proof of the left side, the right side case being symmetric. By Theorem~\ref{thm:Lazard-implies-complete}, the set $F$ is a complete factorization of $A^*$ ordered by $\prec$. Then, since  for all $u,v\in A^+$, we have $u\prec v$ if and only if $\pi(u)\prec^{\pi}\pi(v)$, the set $\pi(F)$ is a complete factorization of $A^*$ ordered by $\prec^\pi$. Thus, we can use Theorem~\ref{thm:Viennot} to see that $\pi(F)$ is a left Lazard set ordered by $\prec^\pi$. Let $f,g\in\pi(F)$ such that $fg\in\pi(F)$. Then $f=\pi(u)$ and $g=\pi(v)$ with $u,v\in F$. Since $\pi(uv)=fg\in\pi(F)$, we obtain that $uv\in F$, and hence that $u\prec uv$. This means that $f\prec^\pi fg$. The result now follows from Theorem~\ref{thm:Viennot}.
\end{proof}

\begin{example}
\label{ex:Lyn-bar-resume}
We resume Example~\ref{ex:Lyn-bar} and we illustrate the fact that $\overline{\Lyn}$ is Viennot in the case of the binary alphabet $\{{\tt 0,1}\}$. We compute the words of $\overline{\Lyn}$ up to length $5$ thanks to the left and right Lazard constructions. Let $C_1=D_1=\{{\tt 0,1}\}$ and, for $i\ge1$, define $C_{i+1} =c_i^* (C_i\setminus\{c_i\})$ (resp., $D_{i+1} =(D_i\setminus\{d_i\})d_i^*$) with $c_i = \min_{<_{\lex}^\pi} (C_i  \cap \{{\tt 0,1}\}^{\le 5})$ (resp., $d_i = \max_{<_{\lex}^\pi} (D_i  \cap \{{\tt 0,1}\}^{\le 5})$). Table~\ref{table:max-left-Lazard} shows the successive steps of the left (resp., right) Lazard procedure for the construction of $\overline{\Lyn}$. The words $c_1, \ldots, c_{14}$ (resp., $d_1, \ldots, d_{14}$) are exactly the words in $\overline{\Lyn}$ of length at most 5. Observe that we have $c_1<_{\lex}^\pi\cdots<_{\lex}^\pi c_{14}$ and $d_1>_{\lex}^\pi\cdots>_{\lex}^\pi d_{14}$. As previously mentioned, we also have $c_i=\max_{<_{\lex}} (C_i  \cap \{{\tt 0,1}\}^{\le 5})$. This is somewhat surprising since in general it is not true that $\max_{<_{\lex}} S=\min_{<_{\lex}^\pi} S$  if $S$ is any subset of $A^*$. 
However, we see that $\min_{<_{\lex}} (D_2\cap \{{\tt 0,1}\}^{\le 5})={\tt 1}\neq\max_{<_{\lex}^\pi} (D_2\cap \{{\tt 0,1}\}^{\le 5})=d_2={\tt 10000}$. Indeed, if we choose the lexicographically minimal words in the right Lazard procedure, then we obtain the Nyldon words; see Table~\ref{table:Nyldon-right-Lazard}. The equality 
\[
	\textstyle 
	\max_{<_{\lex}} (C_i  \cap \{{\tt 0,1}\}^{\le 5})=\min_{<_{\lex}^\pi} (C_i  \cap \{{\tt 0,1}\}^{\le 5})
\]
is due to the fact that the sets $C_i$ are all prefix-free.
\begin{table}[h]
\centering
\[
	\begin{array}{c || l | l||l|l}
	i 	& C_i \cap \{{\tt 0,1}\}^{\le 5} 		& c_i 			& D_i \cap \{{\tt 0,1}\}^{\le 5} 			& d_i 			\\
	\hline 
	1 	& \{{\tt 0,1}\} 						& {\tt 1} 		& \{{\tt 0,1}\} 							& {\tt 0} 		\\
	2 	& \{{\tt 0, 10, 110, 1110, 11110}\} 	& {\tt 11110}	& \{{\tt 1, 10, 100, 1000, 10000}\}	 	& {\tt 10000} 	\\
	3 	& \{{\tt 0, 10, 110, 1110 }\}			& {\tt 1110}		& \{{\tt 1, 10, 100, 1000}\}				& {\tt 1000} 	\\
	4 	& \{{\tt 0, 11100, 10, 110 }\}  		& {\tt 11100} 	& \{{\tt 1, 11000, 10, 100}\}				& {\tt 100} 		\\
	5 	& \{{\tt 0, 10, 110 }\} 				& {\tt 110}		& \{{\tt 1, 1100, 11000, 10, 10100}\}		& {\tt 10100} 	\\
	6 	& \{{\tt 0, 1100,10,11010}\}			& {\tt 11010} 	& \{{\tt 1, 1100, 11000, 10}\}			& {\tt 10} 		\\
	7 	& \{{\tt 0, 1100,10}\}  				& {\tt 1100}		& \{{\tt 1, 110, 11010,1100, 11000}\}	& {\tt 11010} 	\\
	8 	& \{{\tt 0, 11000,10}\} 				& {\tt 11000}	& \{{\tt 1, 110, 1100, 11000}\}			& {\tt 11000} 	\\
	9 	& \{{\tt 0, 10}\} 						& {\tt 10} 		& \{{\tt 1, 110}\}						& {\tt 1100} 	\\
	10 	& \{{\tt 0,100,10100}\}			 	& {\tt 10100}	& \{{\tt 1, 11100, 110}\}					& {\tt 110} 		\\
	11 	& \{{\tt 0,100}\}  					& {\tt 100} 		& \{{\tt 1, 1110, 11100}\}				& {\tt 11100} 	\\
	12 	& \{{\tt 0,1000}\}			 		& {\tt 1000} 	& \{{\tt 1, 1110}\}						& {\tt 1110} 	\\
	13 	& \{{\tt 0,10000}\} 				 	& {\tt 10000}	& \{{\tt 1, 11110}\}						& {\tt 11110} 	\\
	14 	& \{{\tt 0}\} 							& {\tt 0}		& \{{\tt 1}\}								& {\tt 1} 		\\
	\end{array}
\]
\caption{The set $\overline{\Lyn}$ is Viennot.}
\label{table:max-left-Lazard}
\end{table}
\end{example}

Table~\ref{table:extremal-Lazard} summarizes the possible extremal choices with respect to the lexicographic order of the Lazard constructions; see Examples~\ref{ex:Lyndon-left-Lazard}, \ref{ex:Lyndon-right-Lazard}, \ref{ex:Nyldon-right-Lazard} and \ref{ex:Lyn-bar-resume}.
\begin{table}[h]
\centering
\renewcommand{\arraystretch}{1.8}
\begin{tabular}{|c|c||c|c|}
   \hline
  \multicolumn{2}{|c||}{\hspace{.5cm} Left Lazard \hspace{.4cm}\ } 	& \multicolumn{2}{c|}{\hspace{.5cm} Right Lazard \hspace{.4cm}\ } \\
   \hline
   lex min 	& $\mathcal{L}$			& lex min 	& $\mathcal{N}$ \\
   \hline
   lex max 	& $\bar{\mathcal{L}}$ 		& lex max 	& $\mathcal{L}$ \\
   
   \hline
\end{tabular}
\caption{The sets obtained thanks to the four possible lexicographical extremal choices of the Lazard constructions.}
\label{table:extremal-Lazard}
\end{table}

\begin{example}
Similarly, the set $\overline{\Nyl}$ is a right Lazard set ordered by the decreasing lexicographic order $>_{\lex}^\pi$ induced by the order ${\tt m<\cdots<1<0}$ on the letters, where $\pi$ is the same permutation as in Examples~\ref{ex:Lyn-bar} and~\ref{ex:Lyn-bar-resume}. 
\end{example}

\section{Mélançon’s algorithm applied to Nyldon words}
\label{sec:Melancon}

Left Lazard factorizations correspond to Hall sets, or rather {\em left Hall sets}  \cite{Viennot1978}. It is easily checked that right Lazard factorizations correspond to {\em right Hall sets},  as studied in \cite{Melancon1992}. We refer the reader to \cite{Viennot1978,Melancon1992} for the formal definition of  Hall sets, and the fact that the notion of left (resp., right) Hall sets and of left (resp., right) Lazard sets are actually  equivalent. As a consequence of Theorem~\ref{thm:Nyldon-right-Lazard}, the set of Nyldon words is a right Hall set. Therefore, the algorithm of M\'elan\c{c}on can be used to find the unique Nyldon conjugate of any primitive word. In this short section, we recall this algorithm.
Let us also mention here that a left version of  the algorithm of Mélançon was recently used in \cite{Perrin-Reutenauer2018} for computing the factorization corresponding to some new left Lazard set. 

In Algorithm~\ref{algo:conjugate}, $T(i)$ designates the $i$th element of the list $T$ while $T(-i)$ denotes the $(n-i+1)$th element of $T$ if $n$ is the length of $T$. 
\begin{algorithm}
\caption{Compute the Nyldon conjugate of a primitive word.}
\begin{algorithmic}
\REQUIRE $w\in A^+$ primitive
\ENSURE NylC is the Nyldon conjugate of $w$
\STATE $\NylC \leftarrow$ list of letters of $w$,  $T\leftarrow$ list of letters of $w$
\WHILE{length(NylC) $>1$}
\IF{$T(1)=\min_{<_{\lex}}\NylC$ and $T(1)<_{\lex}$ $T(-1)$} 
\STATE $T\leftarrow (T(2),\ldots, T(-2), T(-1)\cdot T(1))$
 \ENDIF
\STATE $i\leftarrow 2$, $j\leftarrow 2$
\WHILE{$j\le \length(T)$}
\WHILE{$i\le \length(T)$ and $T(i)\ne\min_{<_{\lex}}\NylC$}
\STATE $i\leftarrow i+1$
\ENDWHILE
\IF{$i\le \length(T)$ and $T(i)<_{\lex}$ $T(i-1)$}
\STATE $T\leftarrow (T(1),\ldots, T(i-1)\cdot T(i),\ldots,T(-1))$
\ENDIF
\STATE $j\leftarrow i+1$, $i\leftarrow i+1$
\ENDWHILE
\STATE NylC $\leftarrow T$
\ENDWHILE
\RETURN $\NylC$
\end{algorithmic}
\label{algo:conjugate}
\end{algorithm}

In the same manner as for Algorithm~\ref{algo:fac}, we can easily see that the worst case complexity of Mélançon’s algorithm is in $O(n(n-1)/2)$ if $n$ is the length of the input primitive word.

\section{Acknowledgment}

Manon Stipulanti is supported by the FRIA grant 1.E030.16. 
We thank S\'ebastien Labb\'e for bringing to our attention the Mathoverflow post of Grinberg.  
We thank Darij Grinberg for his interest in our results and for his comments on a first draft of this text.
We are also grateful to the reviewers for their numerous interesting comments.

\bibliographystyle{alpha}
\bibliography{Nyldon}
\label{sec:biblio}

\end{document}